\documentclass[12pt]{article}

\usepackage{hyperref}
\hypersetup{
	colorlinks=true,
	linktoc=all,     
	linkcolor=blue,
	citecolor=red,
}
\usepackage{geometry}
\usepackage[usenames,dvipsnames]{color}
\usepackage{amssymb} 
\usepackage{bm}
\usepackage{amsmath,amsfonts,amsthm,mathrsfs}
\numberwithin{equation}{section}
\usepackage{geometry}

\usepackage{appendix}
\usepackage{tikz}
\usepackage[numbers]{natbib}

\usetikzlibrary{positioning}
\usetikzlibrary{arrows}
\newtheorem{remark}{Remark}
\newtheorem{definition}{Definition}[section]

\newtheorem{proposition}{Proposition}[section]

\newtheorem{theorem}{Theorem}[section]
\newtheorem{corollary}{Corollary}[section]
\newtheorem{lemma}{Lemma}[section]

\numberwithin{figure}{section}
\numberwithin{equation}{section}

\newcommand{\R}{\mathbb{R}}

\newcommand{\N}{\mathbb{N}}
\newcommand{\Sph}{\mathbb S}

\newcommand{\Pcal}{\mathcal P}
\newcommand{\Ccal}{\mathcal C}

\newcommand{\MR}{\mathcal M_R}

\newcommand{\dd}{\mathop{}\!\mathrm{d}}

\newcommand{\ddsigma}{\,d\sigma}

\newcommand{\ddx}{\,dx}

\newcommand{\norm}[1]{\left\|#1\right\|}
\newcommand{\abs}[1]{\left|#1\right|}

\DeclareMathOperator{\med}{med}
\DeclareMathOperator{\sgn}{sgn}
\DeclareMathOperator{\supp}{supp}
\DeclareMathOperator{\spanop}{span}

\usepackage{array,booktabs,tabularx}

\usepackage{algorithm}
\usepackage{algorithmic}
\usepackage{bm}

\usepackage{stmaryrd}

\begin{document}

\title{Sharp Sobolev Sandwich and Approximation Rates of Radon-Domain $L^p$ Ridge Integral Spaces for ReLU$^k$ Networks}
\author{Authors}
\author{Juncai He\footnotemark[1] \,\footnotemark[2] \and  Zitong Tian\footnotemark[2]}
\date{}

\maketitle
\renewcommand{\thefootnote}{\fnsymbol{footnote}}
\footnotetext[1]{Yau Mathematical Sciences Center, Tsinghua University, Haidian District, Beijing 100084, China.}
\footnotetext[2]{Qiuzhen College, Tsinghua University, Haidian District, Beijing 100084, China.}

\begin{abstract}
We develop the $L^p$ space and approximation theory for shallow neural networks with $\mathrm{ReLU}^k$ activations. The central object is the Radon-domain $L^p$ space
$\mathcal{R}L^p_k(\Omega)$ containing all functions on a bounded domain
$\Omega$ that admit a ridge integral representation whose coefficient density
belongs to $L^p$ in the Radon domain.  In the Hilbert case
$p=2$, we prove by elementary Fourier analysis that this space recovers the
critical Sobolev space $H^{k+(d+1)/2}(\Omega)$. For general
$1<p<\infty$, the identity becomes a sandwich for Bessel-potential Sobolev spaces.
The sharp gap of each side is exactly the Seeger--Sogge--Stein loss for
the Radon transform as a Fourier integral operator. This also clarifies how the
activation regularity and Radon back-projection jointly produce the regularity.
As an application, we
discretize the integral representation using a deterministic interpolation 
skeleton plus uniform sampling. This yields high-probability $L^p$
approximation rates and the optimal Hilbert rate $\mathcal{O}\!\big(n^{-\frac12-\frac{2k+1}{2d}}\big)$ at $p=2$ for linearized neural networks.
\end{abstract}



\section{Introduction}

Although a great variety of specialized architectures now drive state-of-the-art 
results across machine learning, networks with a single hidden layer remain the 
basic building block of deep learning. We write 
such a shallow network as an element of
\begin{align}
    \Sigma_n^{\sigma}
    := \left\{ \sum_{i=1}^n a_i\,\sigma(w_i\cdot x+b_i)
    \;:\; a_i,b_i\in\R,\ w_i\in\R^d \right\},
    \label{eq:shallow-network}
\end{align}
where $\sigma:\R\to\R$ is a fixed activation function. We focus on the rectified
linear unit $\sigma(x)=\max\{0,x\}$ and its powers 
$\mathrm{ReLU}^k$ activations $\max\{0,x\}^k$ with $k\in\N_0$. For convenience of 
derivatives, we use normalized activations $\sigma_k(t)=t_+^k/k!$.

The basic convergence result for \eqref{eq:shallow-network} is qualitative: the
classical universal approximation
theorem~\cite{cybenko_ApproximationSuperpositionsSigmoidal_1989,%
hornik_ApproximationCapabilitiesMultilayer_1991,%
leshno_multilayerfeedforwardnetworks_1993} asserts that $\bigcup_{n}\Sigma_n^{\sigma}$
is dense in $C(K)$ for every compact $K\subset\R^d$, provided $\sigma$ is not a
polynomial. However, these results do not provide \emph{how many} neurons are
needed to reach a prescribed accuracy, which is the question taken up by quantitative
approximation theory.

Quantitative theory fixes a target class $\mathcal Z$ and asks for the rate at
which $\Sigma_n^{\sigma}$ approximates its members. For a bounded domain
$\Omega\subset\R^d$ we say that $\Sigma_n^{\sigma}$ approximates $\mathcal Z$ at
the $L^p$-rate $\alpha>0$ if
\begin{align*}
    \inf_{f_n\in\Sigma_n^{\sigma}}\norm{f-f_n}_{L^p(\Omega)}= \mathcal O(n^{-\alpha}),
    \qquad \forall f\in\mathcal Z .
\end{align*}
There are many target classes in the literature: Barron spaces, also
known as variation spaces~\cite{barron_universalapproximationbounds_1993,%
e_barronspaceflowinduced_2022,siegel_characterizationvariationspaces_2023},
spectral Barron spaces~\cite{barron_universalapproximationbounds_1993,%
klusowski_approximationcombinationsrelu_2018,siegel_HighorderApproximationRates_2022},
and classical Sobolev spaces~\cite{petrushev_approximationridgefunctions_1998,%
pinkus_approximationtheorymlp_1999,mao_ApproximationRatesShallow_2026}, 
classical smooth functions and H\"older spaces~\cite{lu_DeepNetworkApproximation_2021,%
shen_DeepNetworkApproximation_2020, shen_OptimalApproximationRate_2022}.

\paragraph{Barron spaces, variation spaces and Radon BV spaces.}
The Barron space $\mathcal B^{\sigma}$ consists of the functions admitting an
integral representation against the activations,
\begin{align*}
    \mathcal B^{\sigma}
    = \Big\{ f:\Omega\to\R : f(x)=\int_{\mathcal{M}}\sigma(w\cdot x-b)\,d\mu(w,b)
    \ \text{for a finite Radon measure }\mu \Big\},
\end{align*}
where $\mathcal{M} \subset \mathbb{R}^d \times \mathbb{R}$ is bounded.
These coincide with the variation spaces
of~\cite{e_barronspaceflowinduced_2022,siegel_characterizationvariationspaces_2023}.
A probabilistic argument going back to Maurey~\cite{pisier_RemarquesResultatNon_1981,%
Maurey1973-1974} and Jones~\cite{jones_simplelemmagreedy_1992}, applied in this
setting by~\cite{barron_universalapproximationbounds_1993,%
e_barronspaceflowinduced_2022,bach_breakingcursedimensionality_2017}, yields the
dimension-independent rate $\mathcal O(n^{-1/2})$ for approximating $\mathcal B^{\sigma}$
by $\Sigma_n^{\sigma}$, see the surveys~\cite{devore_nonlinearapproximation_1998,%
devore_neuralnetworkapproximation_2021}. Rates that exploit additional structure
were obtained in~\cite{klusowski_approximationcombinationsrelu_2018,%
makovoz_randomapproximantsneural_1996,li_twolayernetworks_2024}, and sharp
bounds for a range of activations in~\cite{siegel_sharpboundsapproximation_2024,%
siegel_optimalconvergencerates_2022}. In particular, for the $\mathrm{ReLU}^k$
activations, the optimal $L^2$-rate is
\begin{align*}
    \inf_{f_n\in\Sigma_n^{\sigma_k}}\norm{f-f_n}_{L^2(\Omega)}
    = \mathcal O\!\big(n^{-\frac12-\frac{2k+1}{2d}}\big)
    = \mathcal O\!\big(n^{-s_k/d}\big),
    \qquad s_k=\tfrac{d+2k+1}{2},
\end{align*}
which is attained on the $\mathrm{ReLU}^k$ variation
space~\cite{siegel_sharpboundsapproximation_2024}.

To test if a function $u$ belongs to the Barron space, one must find a finite measure $\mu$
such that $u$ admits the integral representation above.
Using the Radon transform,
the Radon-domain bounded-variation spaces $\mathcal{R}\mathrm{BV}^k$
of~\cite{parhi_banachspacerepresenter_2021,parhi_whatkindsfunctions_2022} recast this membership question in the Radon domain, a viewpoint introduced for ReLU
networks by~\cite{ongie_FunctionSpaceView_2019} and developed through the ridgelet
and inverse-Radon analysis in~\cite{sonoda_NeuralNetworkUnbounded_2017,%
unser_RidgesNeuralNetworks_2023,bartolucci_understandingneuralnetworks_2023}.
This function class is equivalent to the variation
definition~\cite{siegel_characterizationvariationspaces_2023}, in which the optimal
rate is likewise recovered~\cite{parhi2022ridge}.
A closely related Banach-space framework was developed in
\cite{unser_KernelMethodsNeural_2024}, where native spaces are built from an
operator and a generic Radon-domain norm, including $L^p$-type Radon-domain norms. 
 Our spaces $\mathcal{R}L^p_k(\Omega)$ follow the same
Radon-domain philosophy, but are specialized to bounded-bias
$\mathrm{ReLU}^k$ ridge integrals on $\Omega$.
The analysis operator used throughout this paper, and the sparsifying
action on a single neuron, is that
of~\cite{parhi_banachspacerepresenter_2021,parhi_whatkindsfunctions_2022}.

\paragraph{Spectral Barron spaces.}
Spectral Barron spaces sit closer to classical Fourier analysis and PDE
theory~\cite{hormander_ExistenceApproximationSolutions_2005,%
choulli_functionalanalysispartial_2025}. Originating
in~\cite{barron_universalapproximationbounds_1993}, the spectral Barron space
$\mathscr B$ collects functions whose Fourier transform has a finite weighted
moment. Approximation rates are studied
in~\cite{klusowski_approximationcombinationsrelu_2018,%
siegel_HighorderApproximationRates_2022,ma_uniformapproximationrates_2022,%
xu_finiteneuronmethod_2020}, embeddings
in~\cite{wu_embeddinginequalitiesbarrontype_2023}, and further structural results
in~\cite{meng_newfunctionspace_2022,liao_spectralbarronspace_2025a}.

\paragraph{Sobolev spaces.}
For classical Sobolev spaces, the benchmark entropy rate is
$n^{-s/d}$~\cite{devore_nonlinearapproximation_1998,
devore_OptimalNonlinearApproximation_1989,pinkus_NWidthsApproximationTheory_1985,
edmunds_FunctionSpacesEntropy_1996}.
Deep ReLU networks attain this rate, up to logarithmic factors and with
matching lower bounds~\cite{siegel_optimalapproximationrates_2023,
yang_optimalapproximationsobolev_2025}. For shallow $\mathrm{ReLU}^k$ networks, 
the natural critical smoothness $s_k = \frac{d+2k+1}{2}$
appears in~\cite{siegel_sharpboundsapproximation_2024,
yang_OptimalRatesApproximation_2025}.
Recent work~\cite{mao_ApproximationRatesShallow_2026} establishes sharp shallow
$\mathrm{ReLU}^k$ rates for Sobolev--Slobodeckij space with
error measured in $L^q(\Omega)$ using Radon-transform estimates.

\paragraph{The critical Sobolev space and integral representations.}
Among Sobolev targets, the space $H^{s_k}(\Omega)$ at the critical smoothness
\begin{align*}
    s_k := k+\frac{d+1}{2}=\frac{d+2k+1}{2}
\end{align*}
is distinguished for the $\mathrm{ReLU}^k$ activations: it is the reproducing
kernel Hilbert space (RKHS) generated by the bounded-bias ridge kernel
\begin{align*}
    K(x,y)=\int_{\MR}\sigma_k(w\cdot x-b)\,\sigma_k(w\cdot y-b)\,d\lambda(w,b),
\end{align*}
with $R>R_0:= \sup_{x\in\Omega}\abs{x}$, and $\MR:=\Sph^{d-1}\times[-R,R], d\lambda:=d\sigma\,db$.
 The spectral computation underlying this
identification goes back to~\cite{bach_breakingcursedimensionality_2017},
who diagonalized the $\mathrm{ReLU}^k$ kernel for the uniform feature
distribution, and was turned into a precise integral representation by~\cite{liu_integralrepresentationssobolev_2025}. Writing
$\tilde x=\binom{x}{1}$ and normalizing the parameters to the sphere
$\theta\in\Sph^{d}\subset\R^{d+1}$, they prove
\begin{align}
    H^{s_k}(\Omega)
    = \Big\{ \int_{\Sph^{d}}\sigma_k(\theta\cdot\tilde x)\,\psi(\theta)\,\dd\theta
    : \psi\in L^2(\Sph^{d}) \Big\},
    \quad
    \norm{f}_{H^{s_k}(\Omega)}\simeq\inf_{\psi}\norm{\psi}_{L^2(\Sph^{d})},
    \label{eq:LMX-representation}
\end{align}
identify $H^{s_k}(\Omega)$ as the $\mathrm{ReLU}^k$ RKHS, and show that
linearized networks with fixed, well-distributed inner parameters
$\{\theta_j^*\}\subset\Sph^{d}$ reach the optimal rate
$O(n^{-\frac12-\frac{2k+1}{2d}})$ in $H^{s_k}(\Omega)$.
\cite{petrushev_approximationridgefunctions_1998} also attains the same rate for
$H^{s_k}(\Omega)$ for a general class of activations containing $\sigma_k$.

The representation \eqref{eq:LMX-representation} is proved on the unit sphere
$\Sph^{d}$, reached by lifting $(w,b)\mapsto\theta=(w,b)/\sqrt{1+\abs{b}^2}$. This
normalization couples the weight $w$ and the bias $b$, which is convenient for
spherical harmonics but at odds with the asymmetric roles the two parameters play
in practice, where $w$ is a direction and $b$ a bounded bias. In this work, we keep them separate, working directly on the bounded-bias cylinder
$\MR=\Sph^{d-1}\times[-R,R]$, and we replace the spherical-harmonic machinery
of~\cite{liu_integralrepresentationssobolev_2025} by some elementary estimates on
the Radon side for the $L^2$-case. Moreover, we also develop the $L^p$ density
space theory and its approximations.

\paragraph{Random feature methods.}
Fixing the inner parameters $\{(w_i,b_i)\}$ at random and optimizing only the outer coefficients turns the nonconvex training of \eqref{eq:shallow-network}
into a convex least-squares problem. This is the random feature
method~\cite{rahimi_randomfeatureslargescale_2007,rahimi_WeightedSumsRandom_2008},
which realizes a kernel as a Monte-Carlo average of feature products. Drawing the
features i.i.d.\ from a fixed (e.g.\ uniform) distribution reproduces the kernel
in expectation, but the resulting approximation error decays only at the
Monte-Carlo rate $n^{-1/2}$~\cite{rahimi_WeightedSumsRandom_2008,%
rudi_GeneralizationPropertiesLearning_2017}. 
\cite{chen_dualityframeworkanalyzing_2025} studies random feature models beyond the Hilbertian regime by a duality framework. 
Recent work~\cite[Theorem 6.2]{liu_integralrepresentationssobolev_2025} provides theoretical
backing for random feature methods in the critical Sobolev space $H^{s_k}(\Omega)$ 
by linearized networks, with rate $\mathcal{O}\left( (n/\log n)^{-s_k/d}\right)$.

\paragraph{Contributions and main results.}
Throughout, $\Omega\subset\R^d$ is a bounded Lipschitz domain admitting a bounded
Sobolev-extension operator, and we fix $R>R_0$ so that the
bias variable stays in $[-R,R]$ over $\Omega$. The basic object is the
bounded-bias ridge activations
\begin{align*}
    \phi^k_\theta(x):=\sigma_k(w\cdot x-b),
    \qquad \theta=(w,b)\in\MR:=\Sph^{d-1}\times[-R,R],
\end{align*}
which we superpose against a density $g$ through the two ridge integrals
\begin{align*}
    \mathrm{S}^k_R g:=\int_{\MR}\phi^k_\theta\,g(\theta)\,d\lambda(\theta),
    \quad
    \mathrm{S}^k g:=\int_{\Sph^{d-1}\times\R}\phi^k_\theta\,g(\theta)\,d\lambda(\theta), \quad d\lambda:=d\sigma\,db,
\end{align*}
over the bounded-bias window $\MR$ and over the full parameter space.
Let $\langle \xi\rangle := (1+|\xi|^2)^{1/2}$. For $s \ge 0$, define
the inhomogeneous Sobolev norm and the homogeneous Sobolev seminorm on $\mathbb R^d$ by
\begin{align*}
    \|u\|_{H^s(\mathbb R^d)}^2
    := (2\pi)^{-d}\int_{\mathbb R^d} \langle \xi\rangle^{2s} |\widehat u(\xi)|^2\,d\xi, \quad
    |u|_{\dot H^s(\mathbb R^d)}^2
    := (2\pi)^{-d}\int_{\mathbb R^d} |\xi|^{2s} |\widehat u(\xi)|^2\,d\xi ,
\end{align*}
For functions $g=g(w,b)$ on $\mathbb S^{d-1}\times\mathbb R$, set
\begin{align*}
    \|g\|_{L^2(\mathbb S^{d-1};H^s(\mathbb R))}^2
    := \int_{\mathbb S^{d-1}}
        \|g(w,\cdot)\|_{H^s(\mathbb R)}^2\,d\sigma(w).
\end{align*}
Let $1<p<\infty$, $(I-\Delta)^{s/2}$ be the Bessel potential,
i.e.\ the Fourier multiplier with symbol $\langle\xi\rangle^{s}$.
Throughout, $W^{s,p}(\mathbb R^d)$ denotes the Bessel--potential Sobolev space
$\{u:(I-\Delta)^{s/2}u\in L^p\}$~\cite{stein_harmonicanalysis_1993,triebel2008function}.
For $p=2$ this is the usual $H^s$.
On a bounded Lipschitz domain $\Omega\subset\R^d$ we use the 
extension~\cite{adams_SobolevSpaces_2008} to define
\begin{align*}
     \norm{u}_{W^{s,p}(\Omega)}
    :=\inf\bigl\{\,\norm{u_e}_{W^{s,p}(\R^d)}:u_e\in W^{s,p}(\R^d),\ u_e|_\Omega=u\,\bigr\}.
\end{align*}
The function space theory relies on the following observation: pairing the
$\tfrac{d-1}{2}$ orders gained by the dual Radon transform $\mathcal R^*$ with
the $k+1$ orders carried by the ridge profile $\sigma_k$ accounts for the full
critical smoothness
\begin{align*}
    s_k=(k+1)+\tfrac{d-1}{2},
\end{align*}
and this is exactly what makes the ridge integral $\mathrm{S}^k_R$ land in $H^{s_k}$. 
When $p=2$, the gain is reachable. For general $1<p<\infty$ it 
is unavailable but carrying the
Seeger--Sogge--Stein~\cite{seeger1991regularity} loss $\delta_p:=(d-1)\abs{\frac1p-\frac12}$,
follows from the $L^p$ regularity of
$\mathcal R$ as a Fourier integral operator.

Fix a bounded open set $\Omega\subset\R^d$ and
$R>R_0:=\sup_{x\in\Omega}|x|$.  For $1<p<\infty$, we define
(see Definition~\ref{def:radon-lp-representation-space})
the \emph{Radon--domain $L^p$ space of order $k$} on $\Omega$ for the $\sigma_k$
activation to be the image of $\mathrm{S}^k_R$ on $L^p(\MR)$,
\begin{align*}
    \mathcal{R}L^p_k(\Omega)
    :=\mathrm{S}^k_R\bigl(L^p(\MR)\bigr)\subset L^p(\Omega),
\end{align*}
equipped with the quotient norm
\begin{align*}
    \norm{u}_{\mathcal{R}L^p_k(\Omega)}
    :=\inf\left\{
    \norm{g}_{L^p(\MR)}:
    g\in L^p(\MR),\ \mathrm{S}^k_R g=u\text{ on }\Omega
    \right\}.
\end{align*}
We prove (Theorem~\ref{thm:bounded-bias-decomposition} and
Corollary~\ref{cor:quotient-norm-equiv-sobolev}) that when $p=2$, $\mathcal{R}L^2_k(\Omega)$ is exactly $H^{s_k}(\Omega)$
by elementary Fourier analysis. This recovers~\eqref{eq:LMX-representation}
of~\cite{liu_integralrepresentationssobolev_2025}. For general $1<p<\infty$, we
prove (Theorem~\ref{the:radon-lp-sobolev-embedding} and
Proposition~\ref{prop:radon-lp-sandwich-sharp}) a sharp sandwich embedding theorem
between $\mathcal{R}L^p_k(\Omega)$ and Bessel-potential Sobolev spaces,
\begin{align*}
    W^{s_{k,p},p}(\Omega)
    \hookrightarrow
    \mathcal{R}L^p_k(\Omega)
    \hookrightarrow W^{r_{k,p},p}(\Omega),
\end{align*}
where
\begin{align*}
    s_{k,p}:=k+\frac{d+1}{2}+(d-1)\left|\frac1p-\frac12\right|,
    \qquad
    r_{k,p}:=k+\frac{d+1}{2}-(d-1)\left|\frac1p-\frac12\right|.
\end{align*}
Note that when $p=2$, $s_{k,2}=r_{k,2}=s_k$ and the sandwich collapses. 
Remark~\ref{rem:slobodeckij-sandwich} is a corollary of the Sobolev--Slobodeckij space
version.

With the above results, we can discretize the integration $\mathrm{S}^k_R g$ and use Monte Carlo sampling to obtain high-probability approximation rates.
Let $\mu_R:=\Lambda_R^{-1}\lambda$ be
the uniform probability measure on $\MR$, with $\Lambda_R:=2R\abs{\Sph^{d-1}}$, and
set $\gamma_{k,p}:=\frac{k+1/p}{d}, \quad \alpha_p:=1-\frac1{\min(p,2)}$.
We attain the approximation rate with high probability
(Theorem~\ref{thm:high-prob-random-feature}): 
for every $u\in\mathcal{R}L^p_k(\Omega)$ and $\delta\in(0,1)$, there exist
coefficients $(a_i)_{i\le N}$ and $(b_j)_{j\le n}$ with $N=\mathcal O(n)$ for which
\begin{align*}
    \Big\|u-\sum_{i=1}^N a_i\,\phi^k_{\zeta_i}-\sum_{j=1}^n b_j\,\phi^k_{\Theta_j}
    \Big\|_{L^p(\Omega)}
    \le
    C\bigl(\log(2/\delta)\bigr)^{\alpha_p}
    \norm{u}_{\mathcal{R}L^p_k(\Omega)}\,
    n^{-\alpha_p-\gamma_{k,p}}
\end{align*}
with probability at least $1-\delta$ over the random neurons $\Theta_j$
drawn i.i.d.\ \emph{uniformly} from $\mu_R$, the
deterministic neurons $\zeta_i$ being independent of $u$.
By the sandwich embedding theorem above, the same rate is attained for
every $u\in W^{s_{k,p},p}(\Omega)$. When $p=2$, the optimal Hilbert rate
$n^{-s_k/d}$ is recovered for every $u \in H^{s_k}(\Omega)$, removing
the log term of~\cite[Theorem 6.2]{liu_integralrepresentationssobolev_2025}. 
See Corollary~\ref{cor:rf-sobolev}.

It is instructive to compare the rate exponent $\alpha_p+\gamma_{k,p}$ with the
two smoothness exponents of the sandwich. A short computation gives
\begin{align*}
    \alpha_p+\gamma_{k,p}=\frac{r_{k,p}}{d}\quad(1<p\le2),
    \qquad
    \frac{r_{k,p}}{d}\le\alpha_p+\gamma_{k,p}<\frac{s_{k,p}}{d}\quad(2<p<\infty).
\end{align*}
Away from $p=2$ the exponent is strictly below $s_{k,p}/d$, and strictly above
$r_{k,p}/d$ once $2<p<\infty$ and $d>2$ indicating that $\mathcal{R}L^p_k(\Omega)$ 
is a genuine intermediate space rather than either endpoint of the sandwich.
See Remark~\ref{rem:lp-rate-comparison}.

\paragraph{Organization.}
Section~\ref{sec:radon-bv-sobolev} develops the $p=2$ theory by elementary Fourier analysis, with a review of the required Radon transform facts and the one-dimensional warm-up case.
Section~\ref{sec:radon-lp-theory} extends the theory to $1<p<\infty$ via
the Fourier integral operator regularity of the Radon transform, leading to the
Sobolev sandwich Theorem. It concludes with an additional microlocal proof of the
embedding and its sharpness.
As an application, Section~\ref{sec:approximation-radon-rf} discretizes
$\mathcal{R}L^p_k(\Omega)$ into a high-probability approximation method.
Section~\ref{sec:numerical-experiments} illustrates the predicted rates
numerically.
Section~\ref{sec:conclusion} concludes the paper with a summary of the 
main findings and a discussion of future work.
Appendix~\ref{app:proofs} collects the postponed proofs, while
Appendix~\ref{app:fio-background} provides the Fourier integral operator
background used in the paper.

\section{Radon-domain density and \texorpdfstring{$L^2$}{L2} theory}
\label{sec:radon-bv-sobolev}

\paragraph{Notation.}
We continue to use the notation of the introduction. In particular
$\sigma_k(t)=t_+^k/k!$, the bounded-bias
cylinder $\MR=\Sph^{d-1}\times[-R,R]$ with $d\lambda(w,b)=d\sigma(w)\,db$, and
the ridge integrals
\begin{align}
    \mathrm{S}^k_R g(x):&=\int_{\MR}\sigma_k(w\cdot x-b)\,g(w,b)\,d\lambda(w,b),\\
    \mathrm{S}^k g(x):&=\int_{\Sph^{d-1}\times\R}\sigma_k(w\cdot x-b)\,g(w,b)\,d\lambda(w,b)
    \label{eq:AR-Sk-recall}
\end{align}
on the bounded window and on the full parameter space.
We write $\mathcal{P}_k$ for the space of polynomials of degree at most $k$.
Throughout $d\ge2$ and $k\in\N_0$.

\paragraph{Constants and rate exponents.}
Let $\Omega \subset \R^d$ be a bounded domain with Lipschitz boundary~\cite{adams_SobolevSpaces_2008},
admitting a bounded extension operator for the Sobolev spaces we use below.
We use the following geometric constants, which depend only on $\Omega$ and $d$:
\begin{align*}
    R_0:=\sup_{x\in\Omega}\abs x,
    \quad R>R_0,
    \quad
    \MR:=\Sph^{d-1}\times[-R,R],
    \quad
    \Lambda_R:=\lambda(\MR)=2R\,\abs{\Sph^{d-1}}.
\end{align*}
For a Sobolev function $u$, we fix a support radius
$A\ge R_0$ with $\supp u\subset B_A(0)$. The Fourier-side constant
is $c_d:=\tfrac{1}{2(2\pi)^{d-1}}$. The smoothness and rate exponents used
throughout are
\begin{align*}
    s_k:&=k+\frac{d+1}{2}=\frac{d+2k+1}{2},
    \quad
    \delta_p:=(d-1)\abs{\frac1p-\frac12}, \nonumber\\
    \rho_p:&=\frac{d-1}{2}-\delta_p=(d-1)\min\!\left(\frac1p,1-\frac1p\right),
    \\
    s_{k,p}:&=s_k+\delta_p,
    \quad
    r_{k,p}:=s_k-\delta_p=k+1+\rho_p, \nonumber\\
    \gamma_{k,p}:&=\frac{k+1/p}{d},
    \quad
    \alpha_p:=1-\frac1{\min(p,2)},
\end{align*}
together with $Q_{d,k}:=\dim\Pcal_k(\R^d)=\binom{d+k}{k}$ and the conjugate
exponent $p':=\tfrac{p}{p-1}$. At $p=2$ every loss vanishes: $\delta_2=0$,
$s_{k,2}=r_{k,2}=s_k$, $\alpha_2=\tfrac12$, and $\alpha_2+\gamma_{k,2}=s_k/d$.
The exponents $\delta_p,\rho_p,s_{k,p},r_{k,p}$ appear in the $L^p$ theory of
Section~\ref{sec:radon-lp-theory}, while $\gamma_{k,p},\alpha_p$ in the
approximation rate of Section~\ref{sec:approximation-radon-rf}.

\paragraph{Radon transform preliminaries.}
We quickly recall some basic facts about the Radon transform,
see~\cite{ludwig_radontransformeuclidean_1966,helgason_integralgeometryradon_2010,
parhi_banachspacerepresenter_2021,parhi_distributionalextensioninvertibility_2024}
for detailed discussions.

We use the Fourier transform on $\R^d$ and the partial Fourier transform in the
bias variable,
\begin{align*}
    \widehat f(\xi) = \mathcal{F} f(\xi)
    &:=
    \int_{\R^d} f(x)\,e^{-ix\cdot\xi}\ddx,
    &
    \mathcal F_b\Phi(w,\omega)
    &:=
    \int_\R\Phi(w,b)\,e^{-ib\omega}\,db ,
\end{align*}
together with three Radon-domain operators: the Radon transform, its dual
(back-projection), and the ramp filter,
\begin{align*}
    \mathcal Rf(w,b):=\int_{\{x:\,w\cdot x=b\}}\!\!f\,ds,
    &\qquad
    \mathcal R^*\Phi(x):=\int_{\Sph^{d-1}}\Phi(w,w\cdot x)\ddsigma(w), \\
    \mathcal F_b\big(\Lambda^{s}\Phi\big)(w,\omega)
    &:=\abs{\omega}^{s}\,\mathcal F_b\Phi(w,\omega).
\end{align*}
We write $\mathcal S(\mathbb S^{d-1}\times\mathbb R)$ for the smooth functions on the
cylinder whose derivatives all decay rapidly in $b$, and
$\mathcal S'(\mathbb S^{d-1}\times\mathbb R)$ for its dual. $\mathcal F_b$ is a bijection
of each.
Throughout, $f\in\mathcal S(\R^d)$; all identities below extend to larger function
spaces by density. Let $(-\Delta)^{s/2}$ be the Fourier
multiplier $\mathcal{F} \left[(-\Delta)^{s/2}f\right](\xi)=\abs{\xi}^{s}\widehat f(\xi)$.
A basic polar-coordinate identity states that if $h(w,\omega)$ is invariant under
$(w,\omega)\mapsto(-w,-\omega)$, then
\begin{align}
    \int_{\Sph^{d-1}}\!\!\int_\R \abs{\omega}^{d-1}h(w,\omega)\,d\omega\,\ddsigma(w)
    =2\int_{\R^d} h\big(\tfrac{\xi}{\abs{\xi}},\abs{\xi}\big)\,d\xi .
    \label{eq:antipodal}
\end{align}

\smallskip\noindent\emph{Fourier slice theorem.}
Foliating $\R^d$ by the hyperplanes $\{w\cdot x=b\}$ and applying Fubini,
\begin{align}
    \mathcal F_b(\mathcal R f)(w,\omega)
    &=\int_\R\!\Big(\int_{w\cdot x=b}\!\!f\,ds\Big)e^{-ib\omega}\,db
    =\int_{\R^d}f(x)\,e^{-i\omega(w\cdot x)}\ddx
    =\widehat f(\omega w).
    \label{eq:Fourier-slice}
\end{align}

\smallskip\noindent\emph{Commutation with the fractional Laplacian.}
Applying \eqref{eq:Fourier-slice} to $(-\Delta)^{s/2}f$ and using $\abs{w}=1$,
$$
    \mathcal F_b\!\Big(\mathcal R\big((-\Delta)^{s/2}f\big)\Big)(w,\omega)
    =\abs{\omega w}^{s}\widehat f(\omega w)
    =\abs{\omega}^{s}\widehat f(\omega w)
    =\mathcal F_b\!\big(\Lambda^{s}\mathcal Rf\big)(w,\omega),
$$
so, by injectivity of $\mathcal F_b$,
\begin{align}\label{eq:radon-laplace-commute}
    \mathcal R\big((-\Delta)^{s/2} f\big)=\Lambda^{s}\big(\mathcal R f\big).
\end{align}

\smallskip\noindent\emph{Filtered back-projection.}
Using Fourier inversion of the partial Fourier transform to
$\Lambda^{d-1}\mathcal Rf$, Fourier slice theorem\eqref{eq:Fourier-slice},
and applying \eqref{eq:antipodal} with $h(w,\omega)=\widehat f(\omega w)e^{i\omega(w\cdot x)}$,
\begin{equation}\label{eq:fbp}
\begin{aligned}
    \mathcal R^*\Lambda^{d-1}\mathcal Rf(x)
    &=\frac1{2\pi}\int_{\Sph^{d-1}}\!\!\int_\R\abs{\omega}^{d-1}
      \widehat f(\omega w)\,e^{i\omega(w\cdot x)}\,d\omega\,\ddsigma(w) \\
    &=\frac1{\pi}\int_{\R^d}\widehat f(\xi)\,e^{i\xi\cdot x}\,d\xi
    =c_d^{-1}f(x).
\end{aligned}
\end{equation}
This gives the inversion formula
$f=c_d\,\mathcal R^*\Lambda^{d-1}\mathcal Rf$.

\smallskip\noindent\emph{Radon--Plancherel isometry.}
By the one-dimensional Plancherel identity in $b$, the slice theorem
\eqref{eq:Fourier-slice}, and \eqref{eq:antipodal} with
$h(w,\omega)=\abs{\widehat f(\omega w)}^2$,
\begin{align}\label{eq:radon-identity}
    \norm{\Lambda^{\frac{d-1}{2}}\mathcal Rf}_{L^2(\Sph^{d-1}\times\R)}^2
    &=\frac1{2\pi}\int_{\Sph^{d-1}\times\R}
      \left|\mathcal F_b\big(\Lambda^{\frac{d-1}{2}}\mathcal Rf\big)(w,\omega)\right|^2
      \,d\omega\,\ddsigma(w) \nonumber\\
    &=\frac1{2\pi}\int_{\Sph^{d-1}\times\R}\abs{\omega}^{d-1}
      \abs{\widehat f(\omega w)}^2\,d\omega\,\ddsigma(w)
    =\frac1{\pi}\int_{\R^d}\abs{\widehat f(\xi)}^2\,d\xi \nonumber\\
    &=\frac{(2\pi)^d}{\pi}\,\norm{f}_{L^2(\R^d)}^2
    =c_d^{-1}\,\norm{f}_{L^2(\R^d)}^2 .
\end{align}

\paragraph{From a sparsifying operator to the canonical density.}
The exponent $k+1$ is already visible in one dimension, where $\sigma_k=t_+^k/k!$
is the Green's function of $\partial_t^{k+1}$:
\begin{align*}
    \partial_t^{\,k+1}\sigma_k=\delta_0.
\end{align*}
In several variables, the analogous role is played 
by the Radon-domain operator~\cite{parhi_banachspacerepresenter_2021}
$\mathrm{R}_k :=c_d\,\partial_b^{\,k+1}\,\Lambda^{d-1}\,\mathcal R$,
which \emph{sparsifies} a single neuron: for a ridge atom
$x\mapsto\sigma_k(w_0\cdot x-b_0)$ with $(w_0,b_0)\in\Sph^{d-1}\times\R$,
\begin{align}
    \mathrm{R}_k \big[\sigma_k(w_0\,\cdot\,-b_0)\big]
    =\tfrac12\Big(\delta_{(w_0,b_0)}+(-1)^{k+1}\delta_{(-w_0,-b_0)}\Big)
    \qquad\text{on }\Sph^{d-1}\times\R,
    \label{eq:sparsify}
\end{align}
a pair of antipodal point masses on the Radon cylinder
\cite{parhi_banachspacerepresenter_2021,parhi_whatkindsfunctions_2022}. The
antipodal sign $(-1)^{k+1}$ reflects the evenness
$\mathcal Rf(w,b)=\mathcal Rf(-w,-b)$ of the Radon transform together with the
$k+1$ slice derivatives.

Following the terminology of frame and wavelet analysis, where an analysis
operator maps a signal to its coefficient representation and a synthesis
operator reconstructs a signal from coefficients~\cite{christensen_IntroductionFramesRiesz_2016}, 
we call 
\begin{align}
    \mathrm A^k:=\mathrm{R}_k =c_d\,\partial_b^{\,k+1}\,\Lambda^{d-1}\,\mathcal R,
    \label{eq:Rk-def}
\end{align}
the \emph{analysis operator} and $\mathrm S^k,\mathrm S^k_R$~\eqref{eq:AR-Sk-recall}
the \emph{synthesis operators}.

Identity \eqref{eq:sparsify} shows that $\mathrm{A}^k$ recovers the parameters of a single
neuron; applied to a general target $u$, it should return a density representing $u$. Read
as the composition \eqref{eq:Rk-def}, however, $\mathrm{A}^k u$ presupposes that
$\mathcal Ru$ is defined, which for $u\in\mathcal S'(\mathbb R^d)$ requires the
distributional theory of the Radon
transform~\cite{parhi_distributionalextensioninvertibility_2024}. That generality is not
needed here: by the slice theorem \eqref{eq:Fourier-slice}, $\mathrm{A}^k$ acts in the
bias variable as the Fourier multiplier $c_d\,(i\omega)^{k+1}|\omega|^{d-1}$, and we take
this multiplier as the definition. The next lemma shows that it is well posed on a class of
targets already containing both the Sobolev functions and the compactly supported
extensions used below; its proof is postponed to
Appendix~\ref{app:proof-of-analysis-wellposed}.

\begin{lemma}
\label{lem:analysis-multiplier-wellposed}
Let $u\in\mathcal S'(\mathbb R^d)$ be such that $\widehat u$ is a measurable function with
\begin{align}
    (1+|\cdot|)^{-N}\,\widehat u\in L^1(\mathbb R^d)
    \qquad\text{for some }N\ge0 .
    \label{eq:analysis-admissible}
\end{align}
Then $m_u(w,\omega):=(i\omega)^{k+1}|\omega|^{d-1}\,\widehat u(\omega w)$ is defined for
a.e.\ $(w,\omega)\in\mathbb S^{d-1}\times\mathbb R$, and there is a unique 
$g\in\mathcal S'(\mathbb S^{d-1}\times\mathbb R)$
with $\mathcal F_b g=c_d\,m_u$. In particular, \eqref{eq:analysis-admissible} holds for every
$u\in H^{s}(\mathbb R^d)$, $s\in\mathbb R$, and for every compactly supported
$u\in\mathcal S'(\mathbb R^d)$.
\end{lemma}

\begin{definition}
\label{def:canonical-density}
Let $u$ satisfy \eqref{eq:analysis-admissible}. The \emph{canonical density} of $u$ is the
distribution $g_u=\mathrm{A}^k u\in\mathcal S'(\mathbb S^{d-1}\times\mathbb R)$ determined
by Lemma~\ref{lem:analysis-multiplier-wellposed},
\begin{align*}
    \mathcal F_b g_u(w,\omega)
    =c_d\,(i\omega)^{k+1}|\omega|^{d-1}\,\widehat u(\omega w),
    \qquad\text{a.e.\ }(w,\omega)\in\mathbb S^{d-1}\times\mathbb R .
\end{align*}
\end{definition}

\begin{remark}
The definition is consistent with \eqref{eq:Rk-def}, since for $u\in\mathcal S(\mathbb R^d)$
the slice theorem \eqref{eq:Fourier-slice} identifies the multiplier with
$c_d\,\partial_b^{k+1}\Lambda^{d-1}\mathcal Ru$. It also reproduces the antipodal symmetry
of \eqref{eq:sparsify}: as $\widehat u(\omega w)$ is invariant under
$(w,\omega)\mapsto(-w,-\omega)$, every canonical density obeys
$g_u(-w,-b)=(-1)^{k+1}g_u(w,b)$.
\end{remark}

The two factors of the multiplier have transparent origins: $(i\omega)^{k+1}$ comes from
$\partial_b^{k+1}$ and $|\omega|^{d-1}$ from $\Lambda^{d-1}$, so $g_u$ inverts the integral
$\mathrm{S}^k$ (Proposition~\ref{prop:density-to-function}). The two operators realize one
principle: on the Radon side an $L^2$ density corresponds to an $H^{s_k}$ function, with the
gain
\begin{align*}
    s_k=(k+1)+\tfrac{d-1}{2}
\end{align*}
splitting into the $k+1$ orders supplied by the ridge profile $\sigma_k$ and the
$\tfrac{d-1}{2}$ orders supplied by the regularity of the dual Radon transform
$\mathcal R^*$. Theorem~\ref{thm:density-sobolev} makes this exact as a global isometry on
$\mathbb R^d$. Proposition~\ref{prop:density-to-function} synthesizes $u$ from $g_u$ via
$\mathrm{S}^k$, and Proposition~\ref{prop:ridge-integral-regularity} localizes the gain to
the bounded window $\mathcal M_R$, yielding the decomposition of
Theorem~\ref{thm:bounded-bias-decomposition} and the norm equivalence of
Corollary~\ref{cor:quotient-norm-equiv-sobolev}.

\paragraph{Warm up: $d=1$ case of sandwich embedding.}
The one-dimensional case already exhibits the whole mechanism. The loss
$\delta_p=(d-1)\abs{1/p-1/2}$ vanishes for every $p$, the two exponents coincide,
$s_{k,p}=r_{k,p}=k+1=s_k$, and the sandwich becomes a single norm equivalence
\begin{align*}
    \mathcal{R}L^p_k(\Omega)=W^{k+1,p}(\Omega).
\end{align*}
To see this, note that $\Sph^{0}=\{\pm1\}$ is two points, so
$\mathcal R^{*}\Phi(x)=\Phi(1,x)+\Phi(-1,-x)$ has order $(d-1)/2=0$ and contributes no
smoothing of its own. Writing $g=(g_+,g_-)$ for the two sheets (each extended by zero
off $[-R,R]$), the ridge integral is a pair of convolutions
\begin{align*}
    \mathrm{S}^k_R g(x)=(\sigma_k\ast g_+)(x)+(\sigma_k\ast g_-)(-x),
\end{align*}
hence a $(k+1)$-fold antiderivative of $g$. As $\sigma_k=t_+^k/k!$ solves
$\partial^{\,k+1}\sigma_k=\delta_0$, differentiating $k+1$ times returns the densities,
which is \eqref{eq:sparsify} for $d=1$, while the intermediate kernels
$\partial^{\,\ell}\sigma_k=\sigma_{k-\ell}$ with $0\le\ell\le k$ stay bounded on the
compact window. Both bounds follow at once, and since no singular integral intervenes
the equivalence in fact holds for every $1\le p\le\infty$.

This also explains the canonical density. When $d=1$ the analysis map
$\mathrm{A}^k =c_d\,\partial_b^{\,k+1}\Lambda^{d-1}\mathcal R$ collapses to
$\partial^{\,k+1}$, so it is literally $(k+1)$-fold differentiation, of which
$\mathrm{A}^k $ is the Radon-domain generalization in higher dimension. Like
$\partial^{\,k+1}$, whose null space is $\Pcal_k(\R^d)$, the operator $\mathrm{A}^k $
annihilates polynomials of degree $\le k$
\cite[Lemma~19]{parhi_banachspacerepresenter_2021}, so a density recovered from a given
$u$ is determined only up to such a polynomial, which is the polynomial term of
Proposition~\ref{prop:density-to-function}. 

In arbitrary dimension the theory is this one-dimensional identity applied one Radon slice at a time. Along each line normal to
$w\in\Sph^{d-1}$ the activation still inverts $\partial_b^{\,k+1}$, and the only
$d$-dependent step is the reassembly by $\mathcal R^{*}$, which supplies the remaining
$(d-1)/2$ orders of $s_k=(k+1)+\tfrac{d-1}{2}$ and, away from $p=2$, the loss
$\delta_p\propto d-1$. Introducing the Radon transform is thus exactly the device that
turns the $d$-dimensional ridge problem into the elementary fact above followed by a
controlled back-projection.

\subsection{An \texorpdfstring{$L^2$}{L2} isometry on \texorpdfstring{$\R^d$}{Rd}}
Using the isometry~\eqref{eq:radon-identity} of $L^2$ functions, we can show an isometry 
result.
\begin{theorem}
\label{thm:density-sobolev}
Let $u\in H^{s_k}(\R^d)$, and let $g_u=\mathrm{A}^k  u$ be the canonical
density of $u$.
Then $g_u\in L^2(\Sph^{d-1}\times\R)$, with the homogeneous isometry
\begin{align}
    \|g_u\|_{L^2(\Sph^{d-1}\times\R)}
    =
    \sqrt{c_d}\,|u|_{\dot H^{s_k}(\R^d)}.
    \label{eq:density-sobolev-main}
\end{align}
\end{theorem}

\begin{proof}
For $u\in\mathcal S(\R^d)$, Definition~\ref{def:canonical-density} together with
$k+d=\frac{d-1}{2}+s_k$ gives
\begin{align*}
    \|g_u\|_{L^2(\Sph^{d-1}\times\R)}^2
    = c_d^2\,\norm{\Lambda^{k+d}\mathcal Ru}_{L^2(\Sph^{d-1}\times\R)}^2
    = c_d^2\,\norm{\Lambda^{\frac{d-1}{2}}\mathcal R\,(-\Delta)^{s_k/2}u}_{L^2(\Sph^{d-1}\times\R)}^2,
\end{align*}
where the second equality is the commutation~\eqref{eq:radon-laplace-commute}.
Applying the Radon--Plancherel isometry~\eqref{eq:radon-identity} with
$f=(-\Delta)^{s_k/2}u$,
\begin{align*}
    \|g_u\|_{L^2(\Sph^{d-1}\times\R)}^2
    = c_d^2\,c_d^{-1}\,\norm{(-\Delta)^{s_k/2}u}_{L^2(\R^d)}^2
    = c_d\,|u|_{\dot H^{s_k}(\R^d)}^2 ,
\end{align*}
which proves~\eqref{eq:density-sobolev-main} by density.
\end{proof}

\subsection{Ridge integral representation for compactly supported functions}
The following technical lemma allows us to integrate $\sigma_k$ against the
density $g_u$ in the whole space $\R^d$. We postpone the proof to the
appendix~\ref{app:proof-of-k-moments}.
\begin{lemma}
\label{lem:canonical-density-weighted-L1}
Let $u\in H^{s_k}(\R^d)$ be compactly supported, and let $g_u=\mathrm{A}^k  u$ be
the canonical density of $u$. Then $(1+|b|)^k g_u\in L^1(\Sph^{d-1}\times\R)$.
\end{lemma}

Integrating the density $g_u$ recovers $u$ up to a polynomial of degree $k$.
\begin{proposition}
\label{prop:density-to-function}
Let $u\in H^{s_k}(\R^d)$ be compactly supported, with canonical density $g_u = \mathrm{A}^k  u$.
Then there exists $p\in\mathcal P_k(\R^d)$ such that
\begin{align*}
    u(x)=\mathrm{S}^k g_u(x)+p(x)
\end{align*}
holds pointwise for every $x\in\R^d$. More compactly, one may have
\begin{align*}
    \mathrm{S}^k\mathrm{A}^k=\mathrm{Id}\ \ 
    \text{on }H^{s_k}(\mathbb{R}^d)\cap\mathcal E'(\mathbb{R}^d),\ \ 
    \mod \mathcal{P}_k(\mathbb{R}^d)
\end{align*}
\end{proposition}
\begin{proof}
By Lemma~\ref{lem:canonical-density-weighted-L1}, for a.e.\ $w$ the function
\begin{align*}
    F(w,t):=\int_\R \sigma_k(t-b)\,g_u(w,b)\,db
\end{align*}
is well defined for every $t$, has at most polynomial growth of degree $k$ in
$t$, and satisfies
\begin{align}
    \mathrm{S}^k g_u=\mathcal R^*F,
    \qquad
    \partial_t^{k+1}F=g_u
    \quad\text{in }\mathcal S'(\Sph^{d-1}\times\R).
    \label{eq:S_k-to-dual-Radon}
\end{align}
Let $F_0:=c_d\,\Lambda^{d-1}\mathcal Ru$, that is,
$\mathcal F_bF_0(w,\omega):=c_d\abs{\omega}^{d-1}\widehat u(\omega w)$. Then
$\partial_t^{k+1}F_0=g_u$ as well, since
$(i\omega)^{k+1}\mathcal F_bF_0=\mathcal F_bg_u$ by
Definition~\ref{def:canonical-density}. The filtered back-projection~\eqref{eq:fbp} gives
$\mathcal R^*F_0=c_d\,\mathcal R^*\Lambda^{d-1}\mathcal Ru=u,\text{ in }\mathcal S'(\R^d).$
The difference $Q:=F-F_0$ satisfies $\partial_t^{k+1}Q=0$, hence is a
polynomial of degree $\le k$ in $t$,
$Q(w,t)=\sum_{j=0}^k a_j(w)\,t^j, a_j\in\mathcal D'(\Sph^{d-1}),$
so that
\begin{align*}
    \mathcal R^*Q(x)
    =\sum_{j=0}^k \langle a_j(w),(w\cdot x)^j\rangle
    \in\mathcal P_k(\R^d).
\end{align*}
Combining $\mathcal R^*F_0=u$ with \eqref{eq:S_k-to-dual-Radon},
\begin{align*}
    \mathrm{S}^k g_u=\mathcal R^*F=\mathcal R^*F_0+\mathcal R^*Q=u+\mathcal R^*Q
    \quad\text{in }\mathcal S'(\R^d),
\end{align*}
hence $u=\mathrm{S}^k g_u+p$ with $p:=-\mathcal R^*Q\in\mathcal P_k(\R^d)$.

Finally $s_k=k+(d+1)/2>d/2$, so $u$ has a continuous representative. By the
weighted $L^1$ bound of Lemma~\ref{lem:canonical-density-weighted-L1}, $\mathrm{S}^k g_u$
is given by a locally absolutely convergent integral and is continuous by dominated
convergence, so the identity holds pointwise for every $x$.
\end{proof}

\subsection{Bounded-bias regularization and
\texorpdfstring{$L^2$}{L2} integral representation of
\texorpdfstring{$H^{k+(d+1)/2}(\Omega)$}{Hk+(d+1)/2(Omega)} }
The dual Radon transform gives a derivative gain of order $(d-1)/2$ when $p=2$.
We prove it by elementary Fourier analysis. Note that the following regularity lemma is stronger
than Lemma~\ref{lem:radon-lp-mapping} when $p=2$, since it requires only regularity in the 
bias variable.

\begin{lemma}
\label{lem:localized-dual-radon}
Let $\eta\in C_c^\infty(\R^d)$, $\psi\in C_c^\infty(\Sph^{d-1}\times\R)$, and
$r\ge0$. Then for every $\Phi\in L^2(\Sph^{d-1};H^r(\R))$,
\begin{align}
    \norm{\eta\,\mathcal R^*(\psi\Phi)}_{H^{r+(d-1)/2}(\R^d)}
    \le C_{\eta,\psi,r}\,\norm{\Phi}_{L^2(\Sph^{d-1};H^r(\R))}.
    \label{eq:localized-dual-radon-estimate}
\end{align}
\end{lemma}

\begin{proof}
By density it suffices to treat $\Phi\in C_c^\infty(\Sph^{d-1}\times\R)$. For each
fixed $w$, multiplication by $\psi(w,\cdot)\in C_c^\infty(\R)$ is bounded on
$H^r(\R)$, with operator norm controlled by finitely many derivatives of $\psi$
uniformly in $w$; hence multiplication by $\psi$ is bounded on
$L^2(\Sph^{d-1};H^r(\R))$. We may therefore prove
\eqref{eq:localized-dual-radon-estimate} with right-hand side
$\norm{\psi\Phi}_{L^2(\Sph^{d-1};H^r(\R))}$ and absorb the multiplier constant at
the end. Write $\Psi:=\psi\Phi$ and $G:=\mathcal F_b\Psi$. Fourier inversion in the
second variable gives
\begin{align*}
    \widehat{\eta\,\mathcal R^*\Psi}(\xi)
    =\frac1{2\pi}\int_{\Sph^{d-1}}\int_\R
    \widehat\eta(\xi-\omega w)\,G(w,\omega)\,d\omega\,d\sigma(w).
\end{align*}
Fix $\xi\in\R^d$ and set $\langle\omega\rangle:=(1+|\omega|^2)^{1/2}$. Factoring the
integrand as
$|\widehat\eta(\xi-\omega w)|^{1/2}\langle\omega\rangle^{-r}
\cdot|\widehat\eta(\xi-\omega w)|^{1/2}\langle\omega\rangle^{r}|G(w,\omega)|$
and applying Cauchy--Schwarz in $(w,\omega)$,
\begin{align}
    |\widehat{\eta\,\mathcal R^*\Psi}(\xi)|^2
    \le\frac1{(2\pi)^2}\,K(\xi)\,J(\xi),
    \label{eq:CS-split}
\end{align}
where
\begin{align*}
    K(\xi):&=\int_{\Sph^{d-1}}\int_\R|\widehat\eta(\xi-\omega w)|
\langle\omega\rangle^{-2r}\,d\omega\,d\sigma(w), \\
    J(\xi):&=\int_{\Sph^{d-1}}\int_\R|\widehat\eta(\xi-\omega w)|
\langle\omega\rangle^{2r}|G(w,\omega)|^2\,d\omega\,d\sigma(w).
\end{align*}
Denote $h(\zeta):=\langle\zeta\rangle^{-2r}|\zeta|^{-(d-1)}$, so that
$K=2\,|\widehat\eta|*h$ by polar coordinates~\eqref{eq:antipodal}. Since
$h\notin L^1(\R^d)$ when $r\le\tfrac12$, we read off the decay of $K$ from the
convolution itself rather than from $\norm{h}_{L^1}$. By Peetre's inequality
$\langle\xi\rangle^{2r+d-1}\le C\langle\xi-\zeta\rangle^{2r+d-1}
\langle\zeta\rangle^{2r+d-1}$ together with
$\langle\zeta\rangle^{2r+d-1}h(\zeta)=\langle\zeta\rangle^{d-1}|\zeta|^{-(d-1)}
\le C\bigl(1+|\zeta|^{-(d-1)}\bigr)$, the factor
$\langle\,\cdot\,\rangle^{2r+d-1}|\widehat\eta|\in L^1\cap L^\infty(\R^d)$
convolved against $1+|\zeta|^{-(d-1)}\in L^\infty+L^1_{\mathrm{loc}}$ is bounded
uniformly in $\xi$. Hence $K(\xi)\le C_0\langle\xi\rangle^{-(2r+d-1)}$ for all
$r\ge0$, with $C_0=C_0(\eta,r,d)$. Inserting this into
\eqref{eq:CS-split} gives
$\langle\xi\rangle^{2r+d-1}|\widehat{\eta\,\mathcal R^*\Psi}(\xi)|^2
\le\tfrac{C_0}{(2\pi)^2}J(\xi)$. Integrating in $\xi$, and applying Tonelli's
theorem together with the translation invariance
$\int_{\R^d}|\widehat\eta(\xi-\omega w)|\,d\xi=\norm{\widehat\eta}_{L^1(\R^d)}$
(independent of $(w,\omega)$), we obtain
\begin{align}
    \int_{\R^d}\langle\xi\rangle^{2r+d-1}
    |\widehat{\eta\,\mathcal R^*\Psi}(\xi)|^2\,d\xi
    \le\frac{C_0\norm{\widehat\eta}_{L^1}}{(2\pi)^2}
    \int_{\Sph^{d-1}}\int_\R
    \langle\omega\rangle^{2r}|G(w,\omega)|^2\,d\omega\,d\sigma(w).
    \label{eq:localized-radon-final}
\end{align}
Since $2\bigl(r+\tfrac{d-1}2\bigr)=2r+d-1$, the left-hand side of
\eqref{eq:localized-radon-final} is a constant multiple of
$\norm{\eta\,\mathcal R^*\Psi}_{H^{r+(d-1)/2}(\R^d)}^2$, while by Plancherel in $b$
the right-hand side is a constant multiple of
$\norm{\Psi}_{L^2(\Sph^{d-1};H^r(\R))}^2$. Recalling $\Psi=\psi\Phi$ and the
multiplier bound $\norm{\psi\Phi}_{L^2(\Sph^{d-1};H^r(\R))}\le
C_\psi\norm{\Phi}_{L^2(\Sph^{d-1};H^r(\R))}$ proves
\eqref{eq:localized-dual-radon-estimate}.
\end{proof}

Applying this lemma, we can show the regularity of Bounded-bias ridge integrals.
\begin{proposition}
\label{prop:ridge-integral-regularity}
Let $\Omega\subset\R^d$ be bounded and let $R>R_0$. For $g\in L^2(\MR)$, define
\begin{align*}
    \mathrm{S}^k_R g(x)
    :=\int_{\MR}\sigma_k(w\cdot x-b)g(w,b)\,d\lambda(w,b).
\end{align*}
Then $\mathrm{S}^k_R g\in H^{s_k}(\Omega)$ and
\begin{align}
    \norm{\mathrm{S}^k_R g}_{H^{s_k}(\Omega)}
    \le C\norm{g}_{L^2(\MR)},
    \label{eq:AR-regularity}
\end{align}
where $C$ depends only on $\Omega$, $R$, $d$ and $k$.
\end{proposition}

\begin{proof}
Choose $\chi\in C_c^\infty((-R,R))$ with $\chi=1$ on $[-R_0,R_0]$. For a.e.
$w\in\Sph^{d-1}$, set
\begin{align*}
    F(w,t):=\int_{-R}^R \sigma_k(t-b)g(w,b)\,db .
\end{align*}
Then, in the sense of distributions on $(-R,R)$,
$\partial_t^{k+1}F(w,\cdot)=g(w,\cdot)$. Moreover, for $0\le\ell\le k$, the kernels
$\partial_t^\ell\sigma_k(t-b)$ are uniformly bounded on
$\supp\chi\times[-R,R]$. Hence
\begin{align*}
    \norm{\chi F(w,\cdot)}_{H^{k+1}(\R)}
    \le C\norm{g(w,\cdot)}_{L^2([-R,R])},
\end{align*}
and integrating in $w$ gives
\begin{align*}
    \norm{\chi F}_{L^2(\Sph^{d-1};H^{k+1}(\R))}
    \le C\norm{g}_{L^2(\MR)}.
\end{align*}
For $x\in\Omega$ one has $|w\cdot x|\le R_0$, hence $\chi(w\cdot x)=1$, so
\begin{align*}
    \mathrm{S}^k_R g(x)=\mathcal R^*(\chi F)(x),
    \qquad x\in\Omega .
\end{align*}
Fix $\eta\in C_c^\infty(\R^d)$ with $\eta\equiv1$ on $\Omega$ and
$\psi\in C_c^\infty(\Sph^{d-1}\times\R)$ with $\psi\equiv1$ on
$\Sph^{d-1}\times\supp\chi$, so that $\psi\,(\chi F)=\chi F$. By
Lemma~\ref{lem:localized-dual-radon} with $r=k+1$ and
$s_k=r+\tfrac{d-1}{2}$,
\begin{align*}
    \norm{\mathrm{S}^k_R g}_{H^{s_k}(\Omega)}
    &\le \norm{\eta\,\mathcal R^*(\psi\,\chi F)}_{H^{s_k}(\R^d)}
    \le C\norm{\chi F}_{L^2(\Sph^{d-1};H^{k+1}(\R))}
    \nonumber\\
    &\le C\norm{g}_{L^2(\MR)},
\end{align*}
which proves \eqref{eq:AR-regularity}.
\end{proof}

\begin{remark}
    It is natural to ask whether a reverse inequality holds for
    Proposition~\ref{prop:ridge-integral-regularity}. In general this is false, due to
    the nontrivial null space of the localized ridge integral
    operator~\cite{unser_RidgesNeuralNetworks_2023,ludwig_radontransformeuclidean_1966}, even if
    $\Omega$ is replaced by the full ball $B_R(0)$. One can construct
    a large class of counterexamples by spherical harmonics in general dimensions.

    For example, take $d=2$ and $k=1$, and write
    $w=(\cos\varphi,\sin\varphi)$. Set
    \begin{align*}
        g(\varphi,b):=\cos(3\varphi)\mathbf 1_{[-R,R]}(b).
    \end{align*}
    For $x\in B_R(0)$, let
    $t=w\cdot x=x_1\cos\varphi+x_2\sin\varphi$. Then
    \begin{align*}
        \int_{-R}^R (t-b)_+\,db
        =\int_{-R}^t(t-b)\,db
        =\frac12(t+R)^2 .
    \end{align*}
    As a function of $\varphi$, $(t+R)^2$ is a trigonometric polynomial with
    angular frequencies at most $2$. It is therefore orthogonal to
    $\cos(3\varphi)$ in $L^2([0,2\pi])$, and hence
    $\mathrm{S}^k_R g(x)=0$ for every $x\in B_R(0)$, although
    $g\not\equiv0$.
\end{remark}

The polynomial part in Proposition~\ref{prop:density-to-function} can be
realized by $\mathrm{S}^k_R$. This lemma is elementary so we postpone the proof
to the appendix~\ref{app:proof-poly-realization}.

\begin{lemma}
\label{lem:polynomial-lp-density}
Let $1\le p<\infty$, let $\Omega\subset\R^d$ be bounded.
  For every $q\in\mathcal P_k(\R^d)$ there
exists $g_q\in L^p(\MR)$, supported in $\Sph^{d-1}\times[-R,-R_0]$, such that
\begin{align*}
    \mathrm{S}^k_R g_q=q\quad\text{on }\Omega,
    \qquad
    \norm{g_q}_{L^p(\MR)}\le C\norm{q}_{L^p(\Omega)},
\end{align*}
with $C$ depending only on $\Omega$, $R$, $d$, $k$, and $p$.
\end{lemma}

Let $\Omega\subset\R^d$ be bounded and admit a bounded extension operator in
$H^{s_k}$, and choose $R>R_0$. Now we can synthesize
the entire Sobolev space $H^{s_k}(\Omega)$ as the image of $L^2(\MR)$ under $\mathrm{S}^k_R$.

\begin{theorem}\label{thm:bounded-bias-decomposition}
The ridge integral $\mathrm{S}^k_R:L^2(\MR)\to H^{s_k}(\Omega)$ is a surjective
bounded linear operator.
\end{theorem}

\begin{proof}
Boundedness is Proposition~\ref{prop:ridge-integral-regularity}:
$\norm{\mathrm{S}^k_R g}_{H^{s_k}(\Omega)}\le C\norm{g}_{L^2(\MR)}$, and linearity
is immediate.

For surjectivity, let $u\in H^{s_k}(\Omega)$. Since $\Omega$ admits a bounded
extension operator, choose a compactly supported $u_e\in H^{s_k}(\R^d)$ with
$u_e|_\Omega=u$ and $\norm{u_e}_{H^{s_k}(\R^d)}\le C\norm{u}_{H^{s_k}(\Omega)}$.
By Proposition~\ref{prop:density-to-function},
$u_e=q_u+\mathrm{S}^k\widetilde g_u$ in $\R^d$, where $q_u\in\mathcal P_k(\R^d)$ and
$\widetilde g_u$ is the canonical density of $u_e$. Split
$\widetilde g_u=g_u+h_u$ with $g_u:=\widetilde g_u\mathbf 1_{\MR}$ and
$h_u:=\widetilde g_u-g_u$; by Theorem~\ref{thm:density-sobolev},
$g_u\in L^2(\MR)$. For $x\in\Omega$ the tail parameter obeys $\abs b>R>\abs x$,
so $\sigma_k(w\cdot x-b)$ is a polynomial of degree $\le k$ in $x$ and
$r_u:=(\mathrm{S}^k h_u)|_\Omega\in\mathcal P_k(\R^d)$. Restricting the synthesis to
$\Omega$,
\begin{align*}
    u=(q_u+r_u)+\mathrm{S}^k_R g_u\quad\text{on }\Omega,
    \qquad q_u+r_u\in\mathcal P_k(\R^d).
\end{align*}
By Lemma~\ref{lem:polynomial-lp-density} there is $g_\star\in L^2(\MR)$ with
$\mathrm{S}^k_R g_\star=q_u+r_u$ on $\Omega$. Hence
$u=\mathrm{S}^k_R(g_u+g_\star)$ with $g_u+g_\star\in L^2(\MR)$, proving
surjectivity.
\end{proof}

\begin{corollary}\label{cor:quotient-norm-equiv-sobolev}
$\mathrm{S}^k_R$ induces a topological isomorphism between $L^2(\MR)/\ker\mathrm{S}^k_R$
and $H^{s_k}(\Omega)$; equivalently
\begin{align}\label{eq:quotient-norm-equiv-sobolev}
    \norm{u}_{H^{s_k}(\Omega)}
    \simeq\inf\big\{\norm{g}_{L^2(\MR)}:g\in L^2(\MR),\ \mathrm{S}^k_R g=u\text{ on }\Omega\big\}.
\end{align}
\end{corollary}

\begin{proof}
By Theorem~\ref{thm:bounded-bias-decomposition}, $\mathrm{S}^k_R$ is a bounded
linear surjection between Banach spaces with closed kernel, so it factors as a
continuous bijection $L^2(\MR)/\ker\mathrm{S}^k_R\to H^{s_k}(\Omega)$, whose
inverse is continuous by the open mapping theorem. The right-hand side of
\eqref{eq:quotient-norm-equiv-sobolev} is exactly the quotient norm.
\end{proof}

\section{\texorpdfstring{$L^p$}{Lp} theory and Sobolev sandwich embedding}
\label{sec:radon-lp-theory}

The exact isometry of Theorem~\ref{thm:density-sobolev} is special to the Hilbert
case: it rests on Plancherel, which has no $L^p$ analogue. The mechanism behind
it is the regularity of the Radon transform
and its dual which survives for $p\neq2$ once we measure it
correctly. The right tool is the local $L^p$ Sobolev regularity of $\mathcal R$
and $\mathcal R^*$ as Fourier integral operators, available precisely on the
open range $1<p<\infty$ and carrying the sharp Seeger--Sogge--Stein loss
$\delta_p$. The exact equality of $p=2$ then degrades by $2\delta_p$ into
the two-sided Sobolev sandwich of
Theorem~\ref{the:radon-lp-sobolev-embedding}. Throughout this section
$1<p<\infty$, and $W^{s,p}$ denotes the Bessel-potential Sobolev space.
The sharp $L^p$ regularity of the Radon transform
(Theorem~\ref{thm:sss-fio-lp}) carries the Seeger--Sogge--Stein loss
\begin{align}
    \delta_p:=(d-1)\left|\frac1p-\frac12\right|,
    \qquad
    \rho_p:=\frac{d-1}{2}-\delta_p
    =(d-1)\min\!\left(\frac1p,\,1-\frac1p\right)>0,
    \label{eq:deltap-rhop-def}
\end{align}
so that $\mathcal R$ and $\mathcal R^*$ gain exactly $\rho_p$ derivatives in $L^p$
(Lemma~\ref{lem:radon-lp-mapping}).  Accordingly, set
\begin{align*}
    s_{k,p}:=s_k+\delta_p=k+\frac{d+1}{2}+\delta_p,
    \qquad
    r_{k,p}:=s_k-\delta_p=k+1+\rho_p.
\end{align*}
Thus $s_{k,2}=r_{k,2}=s_k$, while $s_{k,p}-r_{k,p}=2\delta_p\ge0$ and
$r_{k,p}\ge k+1>0$ for every $p$.  For $p\ge2$ these reduce to
$s_{k,p}=k+d-\tfrac{d-1}{p}$ and $r_{k,p}=k+1+\tfrac{d-1}{p}$, whereas for
$1<p\le2$ one has $\rho_p=\tfrac{d-1}{p'}$, $p'=\tfrac{p}{p-1}$.  The exponent
$s_{k,p}$ controls the analysis map $u\mapsto\mathrm{A}^k  u$
(Theorem~\ref{thm:canonical-density-lp-fio}); the exponent $r_{k,p}$ controls the
synthesis map $g\mapsto\mathrm{S}^k_R g$
(Proposition~\ref{prop:ridge-integral-regularity-lp}).

We study the $L^p$ counterpart of the Radon-domain
BV space $\mathcal{R}\mathrm{BV}^k$
of~\cite{parhi_banachspacerepresenter_2021,parhi_whatkindsfunctions_2022}, the
measure norm being replaced by an $L^p$ norm of the density.
\begin{definition}
\label{def:radon-lp-representation-space}
Fix a bounded open set $\Omega\subset\R^d$ and
$R>R_0$.  For $1<p<\infty$, the \emph{Radon-domain $L^p$
space of order $k$} on $\Omega$ for the $\sigma_k$ activation is
\begin{align*}
    \mathcal{R}L^p_k(\Omega)
    :=\mathrm{S}^k_R\bigl(L^p(\MR)\bigr)\subset L^p(\Omega),
\end{align*}
equipped with the quotient norm
\begin{align}
    \norm{u}_{\mathcal{R}L^p_k(\Omega)}
    :=\inf\left\{
    \norm{g}_{L^p(\MR)}:
    g\in L^p(\MR),\ \mathrm{S}^k_R g=u\text{ on }\Omega
    \right\}.
    \label{eq:radon-lp-quotient-norm}
\end{align}
\end{definition}
The map $\mathrm{S}^k_R:L^p(\MR)\to L^p(\Omega)$ is bounded, by H\"older's
inequality on the finite-measure window $\MR$ and the boundedness of the kernel
$\sigma_k(w\cdot x-b)$ on the compact set $\overline\Omega\times\MR$
(see \eqref{eq:AR-elementary-Lp} below).  Hence $\ker\mathrm{S}^k_R$ is a
\emph{closed} subspace of $L^p(\MR)$, and the quotient norm
\eqref{eq:radon-lp-quotient-norm} endows $\mathcal{R}L^p_k(\Omega)$ with the
natural Banach-space structure: it is isometrically isomorphic to the quotient
$L^p(\MR)/\ker\mathrm{S}^k_R$.  Since $L^p(\MR)$ is separable and, for
$1<p<\infty$, reflexive, the space $\mathcal{R}L^p_k(\Omega)$ is separable and
reflexive as well~\cite{lax2014functional}.

The definition is independent of the choice of $R>R_0$: enlarging $R$ adds only
densities supported at bias $|b|>R_0$, which contribute polynomials of degree
$\le k$ on $\Omega$, and these already lie in $\mathcal{R}L^p_k(\Omega)$ by
Lemma~\ref{lem:polynomial-lp-density}.  When $p=2$, we have proved in
Theorem~\ref{thm:bounded-bias-decomposition} and
Corollary~\ref{cor:quotient-norm-equiv-sobolev} that
$\mathcal{R}L^2_k(\Omega)=H^{s_k}(\Omega)$ with equivalent norms.

\subsection{\texorpdfstring{$L^p$}{Lp} regularity for the Radon transform}
The Hilbertian computation behind Theorem~\ref{thm:density-sobolev} no longer
applies when $p\neq2$, since it rests on Plancherel. What survives is the
\emph{regularity} of the Radon transform and its dual, now read off the
$L^p$-Sobolev scale. The precise tool is the local $L^p$-Sobolev theory of Fourier
integral operators
(FIOs)~\cite{hormander_LagrangianDistributionsFourier_2009,%
sogge_FourierIntegralsClassical_2017,stein_harmonicanalysis_1993}, whose sharp form
is the Seeger--Sogge--Stein theorem and whose loss is $\delta_p$.

\begin{theorem}
\label{thm:sss-fio-lp}
Let $X$ and $Y$ be smooth $d$-dimensional manifolds, and let
$T\in I^\mu(Y,X;\mathcal{C})$ be a properly supported Fourier integral operator whose
canonical relation is locally the graph of a homogeneous canonical transformation.
Then, for every $1<p<\infty$, every $s\in\R$, and every pair of compactly
supported cutoffs $\chi_X,\chi_Y$,
\begin{align*}
    \norm{\chi_YT\chi_X f}_{W^{s-\mu-\delta_p,p}(Y)}
    \le C\,\norm{f}_{W^{s,p}(X)},
    \qquad
    \delta_p:=(d-1)\abs{\frac1p-\frac12}.
\end{align*}
Let $\Pi_{Y\times X}:\mathcal{C}\longrightarrow Y\times X, \,
    (y,\eta;x,\xi)\longmapsto (y,x),$
be the natural projection of the canonical relation to base variables. If $T$ is
elliptic and $d\Pi_{Y\times X}$ has maximal rank $2d-1$ somewhere, then
the loss $\delta_p$ cannot be improved.
\end{theorem}

\begin{proof}
The estimate is the local $L^p$--Sobolev regularity theorem of
Seeger--Sogge--Stein~\cite[Corollary~2.4]{seeger1991regularity}; see also
\cite[Ch.~IX, Sec.~4]{stein_harmonicanalysis_1993} and
\cite[Ch.~6]{sogge_FourierIntegralsClassical_2017}.  The sharpness statement is
the corresponding elliptic maximal-rank sharpness result in
\cite{seeger1991regularity, ruzhansky_SharpnessSeegerSoggeSteinOrders_1999}.
\end{proof}

A single microlocal fact lets us apply Theorem~\ref{thm:sss-fio-lp} here: $\mathcal
R$ and $\mathcal R^*$ are FIOs of order $-(d-1)/2$ whose canonical relations, after
a conic cutoff separating the two antipodal branches, are local canonical graphs.
We state this and defer the verification together with a review of
the FIO preliminary to Appendix~\ref{app:fio-background}.

\begin{lemma}
\label{lem:radon-fio}
Let $d\ge2$ and $Y:=\Sph^{d-1}\times\R$, with $\R^d$ and $Y$ both of dimension $d$.
The Radon transform $\mathcal R:C_c^\infty(\R^d)\to\mathcal D'(Y)$ and its dual
$\mathcal R^*:C_c^\infty(Y)\to\mathcal D'(\R^d)$ are elliptic Fourier integral
operators of H\"ormander order $-(d-1)/2$, with canonical relations $C_{\mathcal R}$
and $C_{\mathcal R}^t$. Each relation has two antipodal branches, and a conic cutoff
removing one branch turns it into a local canonical graph. Consequently, for all
cutoffs $\chi_X\in C_c^\infty(\R^d)$ and $\chi_Y\in C_c^\infty(Y)$, the localized
operators $\chi_Y\mathcal R\chi_X$ and $\chi_X\mathcal R^*\chi_Y$ satisfy hypotheses of Theorem~\ref{thm:sss-fio-lp} with $\mu=-(d-1)/2$, so the
loss there specializes to $\delta_p=(d-1)\abs{1/p-1/2}$.
\end{lemma}

\begin{proof}
See Section~\ref{app:radon-fio-proof}.
\end{proof}

With Theorem~\ref{thm:sss-fio-lp} and Lemma~\ref{lem:radon-fio} in hand, we state
the mapping properties of $\mathcal R$ and $\mathcal R^*$ used below.

\begin{lemma}
\label{lem:radon-lp-mapping}
Let $d\ge2$ and $1<p<\infty$, and let $\delta_p,\rho_p$ be as in
\eqref{eq:deltap-rhop-def}. The order $-(d-1)/2$ of $\mathcal R$ and
$\mathcal R^*$ produces the Sobolev regularity gain
\begin{align}
    \frac{d-1}{2}-\delta_p=\rho_p
    \quad\bigl(=\tfrac{d-1}{p}\text{ when }p\ge2\bigr).
    \label{eq:radon-gain}
\end{align}
That is, for all $\chi\in C_c^\infty(\R^d)$, $\psi\in C_c^\infty(\Sph^{d-1}\times\R)$,
and $s\in\R$,
\begin{align}
    \norm{\psi\,\mathcal R(\chi f)}_{W^{s+\rho_p,p}(\Sph^{d-1}\times\R)}
    &\le C\,\norm{f}_{W^{s,p}(\R^d)},
    \label{eq:radon-forward-lp}\\
    \norm{\chi\,\mathcal R^*(\psi\Phi)}_{W^{s+\rho_p,p}(\R^d)}
    &\le C\,\norm{\Phi}_{W^{s,p}(\Sph^{d-1}\times\R)}.
    \label{eq:dual-radon-forward-lp}
\end{align}
\end{lemma}

\begin{proof}
By Lemma~\ref{lem:radon-fio}, after inserting the cutoffs and passing to a finite
atlas on $\Sph^{d-1}$, both $\psi\mathcal R\chi$ and $\chi\mathcal R^*\psi$ are,
modulo smoothing operators, finite sums of properly supported FIOs of order
$-(d-1)/2$ whose canonical relations are local canonical graphs. The smoothing remainders map $W^{s,p}$ continuously
into $C^\infty$ and are therefore bounded. Theorem~\ref{thm:sss-fio-lp} with
$\mu=-(d-1)/2$ then yields the gain
$s-\mu-\delta_p=s+\tfrac{d-1}{2}-\delta_p=s+\rho_p$, which is exactly
\eqref{eq:radon-forward-lp} and \eqref{eq:dual-radon-forward-lp}.
\end{proof}
It follows the $L^p$ regularity of the analysis operator.
\begin{theorem}
\label{thm:canonical-density-lp-fio}
Let $1<p<\infty$ and fix any support radius $A\ge R_0$. If $u\in C_c^\infty(\R^d)$
with $\supp u\subset B_A(0)$, then the canonical density $g_u=\mathrm{A}^k u$ of
Definition~\ref{def:canonical-density} satisfies
\begin{align}
    \norm{g_u}_{L^p(\Sph^{d-1}\times\R)}
    \le C_{A,d,k,p}\,\norm{u}_{W^{s_{k,p},p}(\R^d)}.
    \label{eq:canonical-density-lp-fio}
\end{align}
Consequently $u\mapsto g_u$ extends continuously to all $u\in W^{s_{k,p},p}(\R^d)$
supported in $B_A(0)$, with \eqref{eq:canonical-density-lp-fio} still valid.
\end{theorem}

\begin{proof}
Set $N:=k+d$, and write $\mathcal H$ for the Hilbert transform in the bias
variable $b$, normalized by
\begin{align}
    \mathcal F_b(\mathcal H\Phi)(w,\omega)
    =-\,i\,\sgn(\omega)\,\mathcal F_b\Phi(w,\omega).
    \label{eq:hilbert-symbol}
\end{align}
The multiplier defining $\mathrm{A}^k $ in Definition~\ref{def:canonical-density}
factors as
\begin{align}
    (i\omega)^{k+1}\abs{\omega}^{d-1}
    =i^{\,1-d}\,(\sgn\omega)^{d-1}\,(i\omega)^N,
    \label{eq:md-factorization}
\end{align}
an exact identity, since $(\sgn\omega)^{d-1}\,\omega^{d-1}=\abs{\omega}^{d-1}$ and
$i^{1-d}i^{d-1}=1$. By \eqref{eq:hilbert-symbol}, the factor $(\sgn\omega)^{d-1}$
is the symbol of $(i\mathcal H)^{d-1}$, which reduces to the identity when $d$ is
odd and to $i\mathcal H$ when $d$ is even. Writing $M_d:=(i\mathcal H)^{d-1}$ and
using the slice identity
$\mathcal F_b(\mathcal Ru)(w,\omega)=\widehat u(\omega w)$, equation
\eqref{eq:md-factorization} becomes
\begin{align*}
    g_u=c_d\,i^{\,1-d}\,M_d\,\partial_b^N\mathcal Ru
    \qquad\text{in }\mathcal S'(\Sph^{d-1}\times\R),
\end{align*}
consistent with the multiplier definition of $\mathrm{A}^k $ used in
Section~\ref{sec:radon-bv-sobolev}. The Hilbert
transform is bounded on $L^p(\R)$ for every $1<p<\infty$; see, e.g.,
~\cite{duoandikoetxea_FourierAnalysis_2000,stein_harmonicanalysis_1993}.
Applying this in the variable $b$ for
each fixed $w$ and integrating in $w$ shows that $M_d=(i\mathcal H)^{d-1}$ is
bounded on $L^p(\Sph^{d-1}\times\R)$, hence
\begin{align}
    \norm{g_u}_{L^p(\Sph^{d-1}\times\R)}
    \le C\,\norm{\partial_b^N\mathcal Ru}_{L^p(\Sph^{d-1}\times\R)} .
    \label{eq:gu-after-multiplier}
\end{align}

Since $\supp u\subset B_A(0)$, the Radon transform $\mathcal Ru$ is supported in
$\Sph^{d-1}\times[-A,A]$. Choose $\zeta\in C_c^\infty(\R^d)$ with $\zeta\equiv1$
on $B_A(0)$, and $\psi\in C_c^\infty(\Sph^{d-1}\times\R)$ with $\psi\equiv1$ on
$\Sph^{d-1}\times[-A,A]$, so that $\mathcal Ru=\psi\,\mathcal R(\zeta u)$. As
$\partial_b$ is a coordinate vector field on $\Sph^{d-1}\times\R$,
\begin{align*}
    \norm{\partial_b^N\mathcal Ru}_{L^p}
    =\norm{\partial_b^N\bigl(\psi\,\mathcal R(\zeta u)\bigr)}_{L^p}
    \le\norm{\psi\,\mathcal R(\zeta u)}_{W^{N,p}(\Sph^{d-1}\times\R)} .
\end{align*}
Because $s_{k,p}+\rho_p=N$ (indeed
$s_{k,p}+\rho_p=(s_k+\delta_p)+(\tfrac{d-1}{2}-\delta_p)=s_k+\tfrac{d-1}{2}=k+d$),
Lemma~\ref{lem:radon-lp-mapping} (with $\chi=\zeta$, $f=u$, $s=s_{k,p}$) gives
\begin{align}
    \norm{\psi\,\mathcal R(\zeta u)}_{W^{N,p}(\Sph^{d-1}\times\R)}
    \le C\,\norm{u}_{W^{s_{k,p},p}(\R^d)} .
    \label{eq:radon-applied-density}
\end{align}
Chaining \eqref{eq:gu-after-multiplier}--\eqref{eq:radon-applied-density} proves
\eqref{eq:canonical-density-lp-fio} for $u\in C_c^\infty$. Both sides are
continuous in the $W^{s_{k,p},p}$ topology, so the bound extends by density to all
$u\in W^{s_{k,p},p}(\R^d)$ supported in $B_A(0)$, and the limiting $g_u$ agrees
with $\mathrm{A}^k u$ distributionally.
\end{proof}
Applying Lemma~\ref{lem:radon-lp-mapping}, we obtain the regularity of dual Radon transform.
\begin{lemma}
\label{lem:localized-dual-radon-lp}
Let $1<p<\infty$ and $\eta\in C_c^\infty(\R^d)$. Then
\begin{align}
    \norm{\eta\,\mathcal R^*\Phi}_{W^{\rho_p,p}(\R^d)}
    \le C_{\eta,p}\,\norm{\Phi}_{L^p(\Sph^{d-1}\times\R)}
    \label{eq:localized-dual-radon-lp}
\end{align}
for all $\Phi\in L^p(\Sph^{d-1}\times\R)$, with $\rho_p$ as in
\eqref{eq:deltap-rhop-def}.
\end{lemma}

\begin{proof}
Pick $A>0$ with $\supp\eta\subset B_A(0)$ and $\psi\in C_c^\infty(\Sph^{d-1}\times
\R)$ with $\psi\equiv1$ on $\Sph^{d-1}\times[-A,A]$. Since $\eta\,\mathcal
R^*\Phi$ only evaluates $\Phi(w,b)$ at $b=w\cdot x$ with $\abs{w\cdot x}\le A$ on
$\supp\eta$, we have $\eta\,\mathcal R^*\Phi=\eta\,\mathcal R^*(\psi\Phi)$.
Estimate \eqref{eq:dual-radon-forward-lp} of Lemma~\ref{lem:radon-lp-mapping} with
$s=0$ and $\chi=\eta$, together with \eqref{eq:radon-gain},
gives~\eqref{eq:localized-dual-radon-lp}.
\end{proof}

\subsection{\texorpdfstring{$L^p$}{Lp} sandwich embedding between Sobolev spaces}
With the regularity results in hand, we can now prove the regularity of the bounded-bias 
synthesis operator.
\begin{proposition}
\label{prop:ridge-integral-regularity-lp}
Let $1<p<\infty$, let $\Omega\subset\R^d$ be bounded and Lipschitz, and let
$R>\sup_{x\in\Omega}|x|$.  For $g\in L^p(\MR)$, define
\begin{align*}
    \mathrm{S}^k_R g(x)=\int_{\MR}\sigma_k(w\cdot x-b)g(w,b)\,d\lambda(w,b).
\end{align*}
Then
\begin{align}
    \mathrm{S}^k_R g\in W^{r_{k,p},p}(\Omega),
    \qquad
    \norm{\mathrm{S}^k_R g}_{W^{r_{k,p},p}(\Omega)}
    \le C\norm{g}_{L^p(\MR)},
    \label{eq:AR-regularity-lp}
\end{align}
where $r_{k,p}=k+1+\rho_p$ and $C$ depends only on $\Omega,R,d,k,p$.
\end{proposition}

\begin{proof}
The elementary estimate
\begin{align}
    \left\|\mathrm{S}^k_R g\right\|_{L^p(\Omega)}
    \le C\left\|g\right\|_{L^p(\mathcal{M}_R)}
    \label{eq:AR-elementary-Lp}
\end{align}
follows from H\"older's inequality on the finite measure set $\mathcal{M}_R$.

Write $Vg(x):=\int_{\mathcal{M}_R}\sigma_k(w\cdot x-b)g(w,b)\,d\lambda(w,b)$ for the
synthesis evaluated on all of $\mathbb{R}^d$, so that $\mathrm{S}^k_Rg=Vg|_\Omega$.
Fix $\eta\in C_c^\infty(B_R(0))$ with $\eta\equiv1$ on a neighbourhood of
$\overline{\Omega}$, and $\chi\in C_c^\infty(\mathbb{R})$ with $\chi\equiv1$ on
$[-R,R]$. Since $\sigma_k^{(j)}=\sigma_{k-j}$ for $0\le j\le k$ and
$\sigma_k^{(k+1)}=\delta_0$, differentiation in $x$ gives, first for smooth $g$
and then by $L^p$ density,
\begin{align}
    \partial_x^\gamma Vg=\mathcal{R}^{*}\Phi_\gamma
    \quad\text{on }B_R(0),
    \qquad
    \Phi_\gamma(w,t):=w^\gamma\int_{-R}^{R}
    \sigma_{k-\left|\gamma\right|}(t-b)\,g(w,b)\,db,
    \label{eq:AR-derivative-identity}
\end{align}
for every $\left|\gamma\right|\le k+1$, with the convention
$\Phi_\gamma=w^\gamma g\,\mathbf{1}_{\mathcal{M}_R}$ when
$\left|\gamma\right|=k+1$. Direct estimates give
\begin{align*}
    \left\|\chi\,\Phi_\gamma\right\|_{L^p(\mathbb{S}^{d-1}\times\mathbb{R})}
    \le C\left\|g\right\|_{L^p(\mathcal{M}_R)},
    \qquad \left|\gamma\right|\le k+1,
\end{align*}
the case $\left|\gamma\right|=k+1$ being immediate. Applying
Lemma~\ref{lem:localized-dual-radon-lp} to each $\chi\Phi_\gamma$,
\begin{align}
    \left\|\eta\,\partial_x^\gamma Vg\right\|_{W^{\rho_p,p}(\mathbb{R}^d)}
    \le C\left\|g\right\|_{L^p(\mathcal{M}_R)},
    \qquad \left|\gamma\right|\le k+1 .
    \label{eq:AR-derivative-bound}
\end{align}
Multiplication by a fixed element of $C_c^\infty(\mathbb{R}^d)$ is bounded on
$W^{\rho_p,p}(\mathbb{R}^d)$, so \eqref{eq:AR-derivative-bound} together with the
Leibniz rule bounds $\partial^\gamma(\eta\,Vg)$ in $W^{\rho_p,p}(\mathbb{R}^d)$ by
$C\left\|g\right\|_{L^p(\mathcal{M}_R)}$ for every $\left|\gamma\right|\le k+1$. The
lifting property of Bessel potential spaces on
$\mathbb{R}^d$~\cite{stein_harmonicanalysis_1993,triebel2008function},
\begin{align*}
    \left\|F\right\|_{W^{k+1+\rho_p,p}(\mathbb{R}^d)}
    \le C\sum_{\left|\gamma\right|\le k+1}
    \left\|\partial^\gamma F\right\|_{W^{\rho_p,p}(\mathbb{R}^d)},
\end{align*}
applied to $F=\eta\,Vg$, and the definition of the restriction norm,
$\left\|\mathrm{S}^k_Rg\right\|_{W^{k+1+\rho_p,p}(\Omega)}
\le\left\|\eta\,Vg\right\|_{W^{k+1+\rho_p,p}(\mathbb{R}^d)}$, prove
\eqref{eq:AR-regularity-lp}.
\end{proof}
The two embeddings below differ by $s_{k,p}-r_{k,p}=2\delta_p$, the round trip of
the Seeger--Sogge--Stein loss through analysis and synthesis. At $p=2$ this gap
closes and the sandwich becomes the norm equivalence
$\mathcal{R}L^2_k(\Omega)=H^{s_k}(\Omega)$ of
Corollary~\ref{cor:quotient-norm-equiv-sobolev}. Now we can prove the Sobolev Sandwich Theorem.
\begin{theorem}
\label{the:radon-lp-sobolev-embedding}
Let $\Omega\subset\R^d$ be bounded and Lipschitz, let $1<p<\infty$ and let
$R>R_0:=\sup_{x\in\Omega}|x|$. Then
\begin{align}
    \mathcal{R}L^p_k(\Omega)
    \hookrightarrow W^{r_{k,p},p}(\Omega),
    \qquad
    \norm{u}_{W^{r_{k,p},p}(\Omega)}
    \le C\norm{u}_{\mathcal{R}L^p_k(\Omega)},
    \label{eq:Bkp-to-basic-sobolev}
\end{align}
and
\begin{align}
    W^{s_{k,p},p}(\Omega)
    \hookrightarrow \mathcal{R}L^p_k(\Omega),
    \qquad
    \norm{u}_{\mathcal{R}L^p_k(\Omega)}
    \le C\norm{u}_{W^{s_{k,p},p}(\Omega)}.
    \label{eq:sobolev-to-Bkp}
\end{align}
\end{theorem}

\begin{proof}
The first embedding is immediate from
Proposition~\ref{prop:ridge-integral-regularity-lp}: if $u=\mathrm{S}^k_R g$ on
$\Omega$, then $\norm{u}_{W^{r_{k,p},p}(\Omega)}\le C\norm{g}_{L^p(\MR)}$, and
taking the infimum over all such $g$ gives \eqref{eq:Bkp-to-basic-sobolev}.

For \eqref{eq:sobolev-to-Bkp} we argue first on $C_c^\infty$ and then pass to the
limit.  Fix $A>0$ with $\Omega\subset B_A(0)$ and let
$u\in C_c^\infty(\R^d)$ with $\supp u\subset B_A(0)$.  Since $u\in H^{s_k}(\R^d)$,
Proposition~\ref{prop:density-to-function} yields $q\in\mathcal P_k(\R^d)$ and the
canonical density $\widetilde g=\mathrm{A}^k u$ with $u=q+\mathrm{S}^k\widetilde g$ on
$\R^d$.  Set $g_0:=\widetilde g\,\mathbf 1_{\MR}$.  By
Theorem~\ref{thm:canonical-density-lp-fio} (valid for $1<p<\infty$),
\begin{align}
    \norm{g_0}_{L^p(\MR)}
    \le\norm{\widetilde g}_{L^p(\Sph^{d-1}\times\R)}
    \le C\norm{u}_{W^{s_{k,p},p}(\R^d)} .
    \label{eq:canonical-density-bound-in-main}
\end{align}
Exactly as in the proof of Theorem~\ref{thm:bounded-bias-decomposition}, for
$x\in\Omega$ the tail parameter obeys $|b|>R>|x|$, so
$\bigl(\mathrm{S}^k(\widetilde g-g_0)\bigr)|_\Omega\in\mathcal P_k(\R^d)$, whence
\begin{align*}
    u=q_0+\mathrm{S}^k_R g_0\quad\text{on }\Omega,
    \qquad q_0\in\mathcal P_k(\R^d).
\end{align*}
By \eqref{eq:AR-elementary-Lp} and \eqref{eq:canonical-density-bound-in-main},
\begin{align}
    \norm{q_0}_{L^p(\Omega)}
    \le\norm{u}_{L^p(\Omega)}+\norm{\mathrm{S}^k_R g_0}_{L^p(\Omega)}
    \le C\norm{u}_{W^{s_{k,p},p}(\R^d)} .
    \label{eq:q0-Lp-bound-main}
\end{align}
By Lemma~\ref{lem:polynomial-lp-density} with exponent $p$, choose
$g_{q_0}\in L^p(\MR)$ with $\mathrm{S}^k_R g_{q_0}=q_0$ on $\Omega$ and
$\norm{g_{q_0}}_{L^p(\MR)}\le C\norm{q_0}_{L^p(\Omega)}$.  Then
$u=\mathrm{S}^k_R(g_0+g_{q_0})$ on $\Omega$, so by
\eqref{eq:canonical-density-bound-in-main} and \eqref{eq:q0-Lp-bound-main},
\begin{align}
    \norm{u}_{\mathcal{R}L^p_k(\Omega)}
    \le\norm{g_0+g_{q_0}}_{L^p(\MR)}
    \le C\norm{u}_{W^{s_{k,p},p}(\R^d)},
    \qquad u\in C_c^\infty(B_A(0)).
    \label{eq:Ccinfty-bound-lp}
\end{align}

Now let $u\in W^{s_{k,p},p}(\Omega)$ be arbitrary and set $u_e:=Eu$, a compactly
supported extension with $\supp u_e\subset B_A(0)$ (enlarging $A$ if necessary)
and $\norm{u_e}_{W^{s_{k,p},p}(\R^d)}\le C\norm{u}_{W^{s_{k,p},p}(\Omega)}$.
Choose $u_n\in C_c^\infty(B_A(0))$ with $u_n\to u_e$ in $W^{s_{k,p},p}(\R^d)$.  By
\eqref{eq:Ccinfty-bound-lp}, $\{u_n|_\Omega\}$ is Cauchy in the Banach space
$\mathcal{R}L^p_k(\Omega)$; let $v$ be its limit.  Since
$\mathcal{R}L^p_k(\Omega)\hookrightarrow W^{r_{k,p},p}(\Omega)\hookrightarrow
L^p(\Omega)$ by \eqref{eq:Bkp-to-basic-sobolev} and $r_{k,p}>0$, while
$u_n|_\Omega\to u_e|_\Omega=u$ in $L^p(\Omega)$, the two limits agree: $v=u$.
Hence $u\in\mathcal{R}L^p_k(\Omega)$ and, by the continuity of the norm,
\begin{align*}
    \norm{u}_{\mathcal{R}L^p_k(\Omega)}
    =\lim_{n}\norm{u_n}_{\mathcal{R}L^p_k(\Omega)}
    \le C\norm{u_e}_{W^{s_{k,p},p}(\R^d)}
    \le C\norm{u}_{W^{s_{k,p},p}(\Omega)},
\end{align*}
which is \eqref{eq:sobolev-to-Bkp}.
\end{proof}

\begin{remark}
\label{rem:relu-k-derivative-count}
The $k+1$ extra derivatives in
Proposition~\ref{prop:ridge-integral-regularity-lp} come from the activation, not
from differentiability of $g$ in the bias variable.  Equivalently, for
\begin{align*}
    F(w,t)=\int_{-R}^R\sigma_k(t-b)g(w,b)\,db,
\end{align*}
one has $\partial_t^{k+1}F=g$ in distributions.  The remaining
$\rho_p$ derivatives are the local $L^p$ gain of the dual Radon FIO: for $p\ge2$
this is $(d-1)/p$, and for $1<p\le2$ it is $(d-1)/p'$.
\end{remark}

\subsection{A microlocal proof of the \texorpdfstring{$L^p$}{Lp} sandwich and its sharpness}
\label{subsec:insightful-proof-microlocal}
The proof of Theorem~\ref{the:radon-lp-sobolev-embedding} above separates the
Radon smoothing from the $k+1$ derivatives carried by the activation.  For readers
comfortable with Fourier integral operators, the same result has a shorter
microlocal interpretation.  
The relevant analysis and synthesis maps carry the \emph{same} phase as the Radon
transform itself, and differ from it only by a multiplier in the bias frequency:
\begin{align}
    \mathrm{A}^k&=c_d\,m_A(D_b)\,\mathcal R,
    &&m_A(\tau):=(i\tau)^{k+1}|\tau|^{d-1},
    \label{eq:Rk-as-composition}\\
    \mathrm{S}^k_R g&=v_g\big|_{\Omega},
    \quad v_g:=\mathcal R^{*}m_S(D_b)\,g^{R},
    &&m_S:=\mathcal F\sigma_k,
    \quad g^{R}:=g\,\mathbf 1_{\mathcal{M}_R},
    \label{eq:AR-as-composition}
\end{align}
where $m(D_b):=\mathcal F_b^{-1}m\,\mathcal F_b$ acts in the bias variable and
$v_g\in\mathcal S'(\mathbb{R}^d)$ is the synthesis of the truncated density on all of
$\mathbb{R}^d$, so that $\mathrm{S}^k_R$ is its restriction to $\Omega$.  Since
$\partial_t^{k+1}\sigma_k=\delta_0$ forces $(i\tau)^{k+1}m_S=1$, we have
$m_S(\tau)=(i\tau)^{-(k+1)}$ for $\tau\neq0$: away from the origin $m_A$ and $m_S$ are
elliptic and positively homogeneous of the opposite degrees $k+d$ and $-(k+1)$.
Inserting them as amplitudes into the oscillatory representation of $\mathcal R$, which
has a single phase variable, the H\"ormander order formula returns
\begin{align*}
    \underbrace{(k+1\, + \, d-1)}_{m_A} \ +\  \underbrace{(-\frac{d-1}{2})}_{\mathcal{R}}=s_k,
    \qquad
    \underbrace{(-\frac{d-1}{2})}_{\mathcal{R}^*} \  +\  \underbrace{-(k+1)}_{m_S} =-s_k ,
\end{align*}
the same as  factorizations $\partial_b^{k+1}\Lambda^{d-1}\mathcal R$ and
$\mathcal R^{*}(\sigma_k*_b\,\cdot\,)$.
The two maps thus have opposite orders, and their composition is of order zero on each
branch.  Each of the two applications of Theorem~\ref{thm:sss-fio-lp} pays its 
own Seeger-Sogge-Stein loss $\delta_p$, and the two
payments add.  This is the whole content of $s_{k,p}-r_{k,p}=2\delta_p$, and of its
vanishing at $p=2$.

\begin{lemma}
\label{lem:inverse-fio-pair}
Let $Y:=\mathbb{S}^{d-1}\times\mathbb{R}$ and $1<p<\infty$.  Modulo operators with
smooth Schwartz kernels, $\mathrm{A}^k$ and $g\mapsto v_g$ are elliptic Fourier integral
operators of orders $s_k$ and $-s_k$, with canonical relations $\mathcal{C}_{\mathcal R}$
and $\mathcal{C}_{\mathcal R}^{t}$; on each branch of Lemma~\ref{lem:radon-fio} these
relations are local canonical graphs.  Consequently, for all
$\chi\in C_c^\infty(\mathbb{R}^d)$ and $\psi\in C_c^\infty(Y)$,
\begin{align}
    \|\psi\,\mathrm{A}^k(\chi f)\|_{L^p(Y)}
    &\le C\,\|f\|_{W^{s_{k,p},p}(\mathbb{R}^d)},
    \label{eq:direct-Rk-bound}\\
    \|\chi\,v_g\|_{W^{r_{k,p},p}(\mathbb{R}^d)}
    &\le C\,\|g\|_{L^p(\mathcal{M}_R)} .
    \label{eq:direct-AR-bound}
\end{align}
\end{lemma}
\begin{proof}
The Schwartz kernels of $\mathrm{A}^k$ and of $g^{R}\mapsto v_g$ are the oscillatory
integrals \eqref{eq:analysis-kernel}--\eqref{eq:synthesis-kernel} of
Appendix~\ref{app:fio-composition}: one phase variable $\tau$, the Radon phase
$\tau(b-w\cdot x)$ resp.\ its negative, and amplitudes $c_d\,m_A$ and $m_S$.  The
splitting \eqref{eq:amplitude-split} writes each amplitude as $a_++a_-+\chi_0m$, where
$a_\pm$ are elliptic classical symbols of order $\deg m$ supported in $\{\pm\tau>0\}$
and $\chi_0m$ is compactly supported.  The $a_\pm$ pieces are therefore elliptic Fourier
integral operators of orders $s_k$ and $-s_k$, each carrying a single branch of
$\mathcal{C}_{\mathcal R}$, hence a local canonical graph. The $\chi_0m$ piece has a smooth
kernel.  Theorem~\ref{thm:sss-fio-lp}, applied branch by branch as in the proof of
Lemma~\ref{lem:radon-lp-mapping} with $(\mu,s)=(s_k,s_{k,p})$ and $(\mu,s)=(-s_k,0)$,
returns the output smoothness $0$ and $s_k-\delta_p=r_{k,p}$, which are
\eqref{eq:direct-Rk-bound} and \eqref{eq:direct-AR-bound}. The smooth-kernel remainders
deduce bounded smooth operators.
\end{proof}

\begin{proof}[Microlocal proof of Theorem~\ref{the:radon-lp-sobolev-embedding}]
The upper embedding \eqref{eq:Bkp-to-basic-sobolev} follows from
\eqref{eq:direct-AR-bound} with $\chi\equiv1$ on $\Omega$, by restriction to $\Omega$
and the infimum over all bounded-window representations $u=\mathrm{S}^k_Rg$.

For the lower embedding, \eqref{eq:direct-Rk-bound} with $\chi\equiv1$ on $B_A(0)$ and
$\psi\equiv1$ on $\mathcal{M}_R$ replaces Theorem~\ref{thm:canonical-density-lp-fio}, and
the remaining(the polynomial tail, Lemma~\ref{lem:polynomial-lp-density},
extension and density) is verbatim as in the proof above.
\end{proof}

We can also prove the sharpness of the sandwich embedding in the Bessel potential scale.
\begin{proposition}
\label{prop:radon-lp-sandwich-sharp}
Under the assumptions of Theorem~\ref{the:radon-lp-sobolev-embedding}, for every
$\varepsilon\in(0,s_{k,p})$ there is no continuous embedding
\begin{align}
    \mathcal{R}L^p_k(\Omega)\hookrightarrow W^{r_{k,p}+\varepsilon,p}(\Omega),
    \label{eq:no-right-improvement}
\end{align}
and no continuous embedding
\begin{align}
    W^{s_{k,p}-\varepsilon,p}(\Omega)\hookrightarrow\mathcal{R}L^p_k(\Omega).
    \label{eq:no-left-improvement}
\end{align}
For $p\neq2$ the gap $s_{k,p}-r_{k,p}=2\delta_p$ is therefore optimal in the
Bessel-potential scale.
\end{proposition}
\begin{remark}
Since both spaces are Banach spaces continuously embedded in
$L^p(\Omega)$, any set inclusion between them would have a closed
graph and hence would be continuous by the closed graph theorem. 
Thus the failure of the
continuous embeddings above also exclude the corresponding set
inclusions.
\end{remark}

\begin{proof}
For $p=2$ this is the non-improvability of the Sobolev scale, since
$\mathcal{R}L^2_k(\Omega)=H^{s_k}(\Omega)$ with equivalent norms
(Corollary~\ref{cor:quotient-norm-equiv-sobolev}).  Let $p\neq2$.  By
Section~\ref{app:radon-fio-proof} the natural projection of
$\mathcal{C}_{\mathcal R}$ has maximal rank $2d-1$ at every point of each branch, and by
Lemma~\ref{lem:inverse-fio-pair} the analysis and synthesis maps are elliptic Fourier
integral operators of orders $s_k$ and $-s_k$ carrying those relations. The sharpness
part of Theorem~\ref{thm:sss-fio-lp} therefore applies to each.  Validity of
\eqref{eq:no-right-improvement} would produce the forbidden bound
$L^p\to W^{r_{k,p}+\varepsilon,p}$ for the synthesis operator, and validity of
\eqref{eq:no-left-improvement}, composed with the $L^p$-bounded operator
$\psi\,\mathrm{A}^k\eta\,\mathrm{S}^k_R$, the forbidden bound
$W^{s_{k,p}-\varepsilon,p}\to L^p$ for $\mathrm{A}^k$.  Both reductions, the choice of
cutoffs, and the transfer between the quotient norm and the Sobolev norms on $\Omega$
are carried out in Appendix~\ref{app:sharpness}.
\end{proof}

\begin{remark}
\label{rem:slobodeckij-sandwich}
We have developed the sandwich embedding for the Bessel potential Sobolev spaces $W^{s,p}$, 
which in the Littlewood--Paley scale is $F^{s}_{p,2}$. The Sobolev--Slobodeckij space has 
equivalent norm to the Besov space $B^{s}_{p,p}$ for $s>0$, $s\notin\mathbb{N}$
\cite{triebel_theoryfunctionspaces_1992}. Write
$W^{s,p}_{\mathrm{Sl}}(\Omega):=B^{s}_{p,p}(\Omega)$. Recall elementary embeddings~
\cite{edmunds_FunctionSpacesEntropy_1996, triebel_theoryfunctionspaces_1992,shi_NewEstimatesRychkovs_2024} 
\begin{align*}
    B^{s}_{p,\min(p,2)}\hookrightarrow F^{s}_{p,2}\hookrightarrow B^{s}_{p,\max(p,2)},
    \qquad
    B^{s+\varepsilon}_{p,q_{0}}\hookrightarrow B^{s}_{p,q_{1}}
    \quad(\varepsilon>0,\ 0<q_{0},q_{1}\le\infty).
\end{align*}
Inserting them into Theorem~\ref{the:radon-lp-sobolev-embedding} gives, 
for every $\varepsilon>0$,
\begin{align*}
    W^{s_{k,p}+\varepsilon,p}_{\mathrm{Sl}}(\Omega)
    \hookrightarrow\mathcal{R}L^{p}_{k}(\Omega)
    \hookrightarrow W^{r_{k,p}-\varepsilon,p}_{\mathrm{Sl}}(\Omega),
\end{align*}
where $\varepsilon$ may be taken to be $0$ on the left when $p\le2$ and on the right
when $p\ge2$, according to whether $\min(p,2)$ or $\max(p,2)$ equals $p$. Applying 
Proposition~\ref{prop:radon-lp-sandwich-sharp} with $\varepsilon/2$
in place of $\varepsilon$ shows that neither $s_{k,p}$ nor $r_{k,p}$ order can be moved 
by positive amounts.
\end{remark}

\section{High-probability approximation for
\texorpdfstring{$1<p<\infty$}{1<p<infty}}
\label{sec:approximation-radon-rf}
Using the above space theory, we now discretize the integral representation
and prove a high-probability approximation result.
We continue to use the notation of Sections~\ref{sec:radon-bv-sobolev}
and~\ref{sec:radon-lp-theory}: $\sigma_k(t)=t_+^k/k!$,
$s_k=\tfrac{d+2k+1}{2}$, $\MR=\Sph^{d-1}\times[-R,R]$,
$d\lambda(w,b)=d\sigma(w)\,db$, and $\Omega\subset\R^d$ is a bounded Lipschitz
domain with $R>\sup_{x\in\Omega}\abs x$ fixed. Write
\begin{align*}
    \phi^k_\theta(x):=\sigma_k(w\cdot x-b),
    \qquad\theta=(w,b)\in\MR,
\end{align*}
for the bounded-bias ridge dictionary, and let
\begin{align*}
    \mu_R:=\Lambda_R^{-1}\lambda,
    \qquad
    \Lambda_R:=\lambda(\MR)=2R\,\abs{\Sph^{d-1}},
\end{align*}
be the uniform probability measure on $\MR$. Recall that
\begin{align*}
    Q_{d,k}:=\dim\Pcal_k(\R^d)=\binom{d+k}{k},
    \qquad
    \gamma_{k,p}:=\frac{k+1/p}{d},
    \\
    \alpha_p:=\min\left\{\frac12,\frac1{p'}\right\}
    =1-\frac1{\min(p,2)},
    \qquad p':=\frac{p}{p-1}.
\end{align*}
Thus $\alpha_p=\tfrac12$ for $2\le p<\infty$, while
$\alpha_p=\tfrac1{p'}=1-\tfrac1p$ for $1<p<2$. In particular
$\gamma_{k,2}=\tfrac{2k+1}{2d}=\gamma_k$ and
$\alpha_2+\gamma_{k,2}=\tfrac12+\gamma_k=s_k/d$.

The space theory of Section~\ref{sec:radon-lp-theory} is valid on the full range
$1<p<\infty$. The sampling step, however, depends on the Rademacher type of
$L^p(\Omega)$, which equals $\min(p,2)$ with type constant depending only on $p$
\cite{ledoux_ProbabilityBanachSpaces_1991}. Accordingly, the theorem
below has the sampling exponent $\alpha_p=1-1/\min(p,2)$. For $p\ge2$ this is the
usual Hilbertian Monte-Carlo exponent $\tfrac12$; for $1<p<2$ it becomes the
type-$p$ exponent $\tfrac1{p'}$. We write
$\mathcal{R}L^p_k(\Omega)$ for the Radon-domain $L^p$ space of order $k$ of
Definition~\ref{def:radon-lp-representation-space}; recall that for $p=2$ one has
$\mathcal{R}L^2_k(\Omega)=H^{s_k}(\Omega)$ with equivalent norms
(Corollary~\ref{cor:quotient-norm-equiv-sobolev}). The result of this section
discretizes the bounded-bias decomposition of Sections~\ref{sec:radon-bv-sobolev}
and~\ref{sec:radon-lp-theory} into a mixed feature model: $\mathcal O(n)$ deterministic
neurons together with $n$ i.i.d.\ uniform neurons achieve the rate
$n^{-\alpha_p-\gamma_{k,p}}$ with high probability. For $p=2$ this is the optimal
$n^{-s_k/d}$. Remark~\ref{rem:lp-rate-comparison} is a comparison with the Sobolev 
exponent for $p\neq2$.

\begin{theorem}
\label{thm:high-prob-random-feature}
There exist $L,C>0$, depending only on $\Omega,R,d,k,p$, such that for every
$n\in\N$ there are deterministic points $\zeta_1,\dots,\zeta_N\in\MR$ with
$N\le Ln$, depending only on $n$, with the following property. Let
$\Theta_1,\dots,\Theta_n\overset{\mathrm{iid}}\sim\mu_R$, let
$u\in\mathcal{R}L^p_k(\Omega)$, and let $\delta\in(0,1)$. Then there exist
coefficients $(a_i)_{i\le N}$ and $(b_j)_{j\le n}$, depending on $u$, $\delta$,
and the samples, such that
\begin{align}
    \Big\|u-\sum_{i=1}^N a_i\,\phi^k_{\zeta_i}-\sum_{j=1}^n b_j\,\phi^k_{\Theta_j}
    \Big\|_{L^p(\Omega)}
    \le
    C\bigl(\log(2/\delta)\bigr)^{\alpha_p}
    \norm{u}_{\mathcal{R}L^p_k(\Omega)}\
    n^{-\alpha_p-\gamma_{k,p}}
    \label{eq:lp-rf-bound-Bkp}
\end{align}
holds with probability at least $1-\delta$.
\end{theorem}
\begin{remark}
    \label{rem:per-target-statement}
    In Theorem~\ref{thm:high-prob-random-feature} the target $u$ is fixed before the
    neurons $\Theta_j$ are drawn, so the event on which \eqref{eq:lp-rf-bound-Bkp}
    holds may depend on $u$, though the outer parameters can be solved by convex optimization
    with fixed inner parameters. The theorem therefore produces, for each $u$, an
    approximant from $\mathcal{O}(n)$ elements of the dictionary
    $\{\phi^k_\theta\}_{\theta\in\mathcal{M}_R}$, and not a single linear subspace of dimension
    $\mathcal{O}(n)$ serving the whole unit ball. 
\end{remark}
\begin{corollary}
\label{cor:rf-sobolev}
Assume that $\Omega$ admits a bounded compactly supported
$W^{s_{k,p},p}$-extension operator, as in
Theorem~\ref{the:radon-lp-sobolev-embedding}. Then every
$u\in W^{s_{k,p},p}(\Omega)$ satisfies
\begin{align}
    \Big\|u-\sum_{i=1}^N a_i\,\phi^k_{\zeta_i}-\sum_{j=1}^n b_j\,\phi^k_{\Theta_j}
    \Big\|_{L^p(\Omega)}
    \le
    C\bigl(\log(2/\delta)\bigr)^{\alpha_p}
    \norm{u}_{W^{s_{k,p},p}(\Omega)}\,
    n^{-\alpha_p-\gamma_{k,p}}
    \label{eq:lp-rf-bound-sobolev}
\end{align}
holds with probability at least $1-\delta$.
In the Hilbert case $p=2$ one has $\mathcal{R}L^2_k(\Omega)=H^{s_k}(\Omega)$,
$\gamma_{k,2}=\gamma_k$, and $\alpha_2+\gamma_k=s_k/d$, so
\eqref{eq:lp-rf-bound-Bkp} is that for every $u\in H^{s_k}(\Omega)$,
\begin{align}
    \Big\|u-\sum_{i=1}^N a_i\,\phi^k_{\zeta_i}-\sum_{j=1}^n b_j\,\phi^k_{\Theta_j}
    \Big\|_{L^2(\Omega)}
    \le
    C\sqrt{\log(2/\delta)}\,\norm{u}_{H^{s_k}(\Omega)}\,n^{-s_k/d}
    \label{eq:main-rf-bound}
\end{align}
holds with probability at least $1-\delta$.
\end{corollary}

\begin{proof}
By Theorem~\ref{the:radon-lp-sobolev-embedding},
$\norm{u}_{\mathcal{R}L^p_k(\Omega)}\le C\norm{u}_{W^{s_{k,p},p}(\Omega)}$, so
\eqref{eq:lp-rf-bound-sobolev} follows from \eqref{eq:lp-rf-bound-Bkp}. For
$p=2$, $\mathcal{R}L^2_k(\Omega)=H^{s_k}(\Omega)$ with equivalent norms by
Corollary~\ref{cor:quotient-norm-equiv-sobolev}, and
$\alpha_2+\gamma_k=s_k/d$ gives \eqref{eq:main-rf-bound}.
\end{proof}

The proof of Theorem~\ref{thm:high-prob-random-feature} combines two ingredients.
First, a \emph{deterministic skeleton}
(Proposition~\ref{prop:deterministic-skeleton}): piecewise Lagrange interpolation
in the parameter $\theta$ approximates every dictionary element by $\mathcal O(n)$ fixed
neurons with uniform $L^p(\Omega)$ error $\mathcal O(n^{-\gamma_{k,p}})$, the exponent
$\gamma_{k,p}=(k+\tfrac1p)/d$ reflecting the $C^{k,1/p}$ parameter smoothness of
the dictionary. Second, a \emph{variance-reduced sampling lemma} in the type
$\min(p,2)$ space $L^p(\Omega)$
(Lemma~\ref{lem:high-prob-uniform-hilbert}), which discretizes the residual
integral by the $n$ random neurons and contributes the sampling factor
$n^{-\alpha_p}$. Multiplying the two rates gives
$n^{-\alpha_p-\gamma_{k,p}}$; since
$\mathcal{R}L^p_k(\Omega)=\mathrm{S}^k_R(L^p(\MR))$ already absorbs the polynomial
part, no separate polynomial term has to be carried. The skeleton-plus-sampling
strategy is in the spirit of~\cite{siegel_sharpboundsapproximation_2024}.

\subsection{A deterministic skeleton on the parameter cylinder}

For a Banach space $H$, a cube $Q\subset\R^d$, and $\alpha\in(0,1]$, let
$C^{k,\alpha}(Q;H)$ denote the maps $F:Q\to H$ whose $H$-valued Fr\'echet
derivatives up to order $k$ exist and are bounded on $Q$, equipped with the
seminorm
\begin{align*}
    \abs{F}_{C^{k,\alpha}(Q;H)}
    :=\max_{\abs\beta=k}\,
    \sup_{\eta\ne\eta'\in Q}
    \frac{\norm{D^\beta F(\eta)-D^\beta F(\eta')}_H}{\abs{\eta-\eta'}^\alpha}.
\end{align*}
Using the smoothness of the parameters, we have
\begin{lemma}
\label{lem:local-parameter-smoothness}
Let $1<p<\infty$, $\Ccal:=[-1,1]^d$, and let $\varphi:\Ccal\to\MR$ be a smooth
chart. Then $P(\eta):=\phi^k_{\varphi(\eta)}$ belongs to
$C^{k,1/p}(\Ccal;L^p(\Omega))$, with $\abs{P}_{C^{k,1/p}}$ and
$\max_{\abs\beta\le k}\sup_\Ccal\norm{D^\beta P}_{L^p(\Omega)}$ bounded by a
constant depending only on $\Omega$, $R$, $d$, $k$, $p$, and the chart.
\end{lemma}

\begin{proof}
Write $\varphi(\eta)=(w(\eta),b(\eta))$ and $t(x,\eta):=w(\eta)\cdot x-b(\eta)$.
Since $\sigma_k^{(j)}=\sigma_{k-j}$ for $0\le j\le k$, the chain rule gives, for
$1\le\abs\beta\le k$,
\begin{align}
    D^\beta_\eta\,\sigma_k(t(x,\eta))
    =\sum_{j=1}^{\abs\beta} q_{\beta,j}(x,\eta)\,\sigma_{k-j}(t(x,\eta)),
    \label{eq:faa-di-bruno-ridge}
\end{align}
where each $q_{\beta,j}$ is a polynomial in $x$ with coefficients smooth in
$\eta$; all these factors, and the truncated powers $\sigma_{k-j}$ with
$0\le k-j\le k$, are bounded on $\Omega\times\Ccal$. Difference quotients of the
right-hand side converge in $L^p(\Omega)$ by dominated convergence (the
exceptional set $\{x\in\Omega:t(x,\eta)=0\}$ is a hyperplane slice, of measure
zero), so $D^\beta P$ exists as an $L^p(\Omega)$-valued derivative, is given by
\eqref{eq:faa-di-bruno-ridge}, and is bounded.

For the H\"older bound, fix $\abs\beta=k$. When $k=0$, this means estimating
$P(\eta)=\mathbf 1_{\{t(\cdot,\eta)>0\}}$ itself, and the slab estimate below
directly gives the required $1/p$-H\"older continuity. Assume now $k\ge1$. In
\eqref{eq:faa-di-bruno-ridge} every term with $j\le k-1$ involves
$\sigma_{k-j}$ with $k-j\ge1$ and is therefore Lipschitz in $\eta$, uniformly over
$x\in\Omega$. The only term requiring care is $j=k$, which contains the indicator
$\sigma_0(t)=\mathbf 1_{\{t>0\}}$. On the symmetric difference
$\{t(\cdot,\eta)>0\}\,\triangle\,\{t(\cdot,\eta')>0\}$ one has
$\abs{t(x,\eta)}\le\abs{t(x,\eta)-t(x,\eta')}\le C\abs{\eta-\eta'}$ for
$x\in\Omega$; since a slab of width $\varepsilon$ meets the bounded set $\Omega$
in measure $O(\varepsilon)$,
\begin{align*}
    &\norm{\mathbf 1_{\{t(\cdot,\eta)>0\}}-\mathbf 1_{\{t(\cdot,\eta')>0\}}}_{L^p(\Omega)}^p \\
    \le &\big|\{x\in\Omega:\abs{w(\eta)\cdot x-b(\eta)}\le C\abs{\eta-\eta'}\}\big| \\
    \le &C'\abs{\eta-\eta'}.
\end{align*}
Thus the indicator term is $\tfrac1p$-H\"older in $L^p(\Omega)$. Combining the
Lipschitz factors with this $\tfrac1p$-H\"older factor yields
$\abs{P}_{C^{k,1/p}(\Ccal;L^p(\Omega))}<\infty$.
\end{proof}
We need a Function-valued Bramble--Hilbert estimate as following.
\begin{lemma}
\label{lem:banach-bramble-hilbert}
Let $\mathcal{C}=[-1,1]^d$ be the reference cube and let
$y_1,\dots,y_{Q_{d,k}}\in\mathcal{C}$ be unisolvent for $\mathcal{P}_k(\mathbb{R}^d)$,
with associated Lagrange basis
$\ell_1,\dots,\ell_{Q_{d,k}}\in\mathcal{P}_k(\mathbb{R}^d)$.Let $Q\subset\R^d$ be a cube of side $h$ with affine map $A_Q:\Ccal\to Q$, and transport the nodes and basis by
$y_i^Q:=A_Q(y_i)$, $\ell_i^Q:=\ell_i\circ A_Q^{-1}$. Then for every Banach space
$H$, every $\alpha\in(0,1]$, and every $F\in C^{k,\alpha}(Q;H)$,
\begin{align}
    \sup_{\eta\in Q}
    \Big\|F(\eta)-\sum_{i=1}^{Q_{d,k}}\ell_i^Q(\eta)\,F(y_i^Q)\Big\|_H
    \le
    C\,h^{k+\alpha}\,\abs{F}_{C^{k,\alpha}(Q;H)},
    \label{eq:banach-BH}
\end{align}
where $C$ depends only on $d$, $k$, $\alpha$, and the reference nodes.
\end{lemma}

\begin{proof}
The pullback $G:=F\circ A_Q$ satisfies
$\abs{G}_{C^{k,\alpha}(\Ccal;H)}\le(h/2)^{k+\alpha}\abs{F}_{C^{k,\alpha}(Q;H)}$ by
the chain rule, and both sides of \eqref{eq:banach-BH} are invariant under this
pullback, so we may assume $Q=\Ccal$. Let $IG:=\sum_i\ell_i\,G(y_i)$ and let $T$
be the $H$-valued Taylor polynomial of $G$ of degree $k$ at the origin. Taylor's
theorem with integral remainder gives
$\sup_\Ccal\norm{G-T}_H\le C\abs{G}_{C^{k,\alpha}}$. Since $\deg T\le k$ and the
nodes are unisolvent, $IT=T$, hence $G-IG=(G-T)-I(G-T)$ and
\begin{align*}
    \sup_\Ccal\norm{G-IG}_H
    \le(1+\Lambda)\sup_\Ccal\norm{G-T}_H
    \le C\abs{G}_{C^{k,\alpha}},
\end{align*}
where $\Lambda:=\sup_\Ccal\sum_i\abs{\ell_i}$ is the Lebesgue constant of the
reference nodes.
\end{proof}

By polynomial interpolation, we can approximate $\phi_{\theta}$ uniformly over
$\mathcal{M}_R$ using deterministic skeleton on $\MR$.
\begin{proposition}
\label{prop:deterministic-skeleton}
Let $1<p<\infty$. There exist $L,C>0$, depending only on $\Omega$, $R$, $d$,
$k$, and $p$, such that for every $n\in\N$ one can choose points
$\zeta_1,\dots,\zeta_N\in\MR$ with $N\le Ln$, and bounded measurable functions
$\alpha_1,\dots,\alpha_N:\MR\to\R$, such that
\begin{align}
    \sup_{\theta\in\MR}
    \Big\|\phi^k_\theta-\sum_{i=1}^N\alpha_i(\theta)\,\phi^k_{\zeta_i}
    \Big\|_{L^p(\Omega)}
    \le C\,n^{-\gamma_{k,p}}.
    \label{eq:skeleton-error}
\end{align}
\end{proposition}

\begin{proof}
Cover $\Sph^{d-1}$, hence the cylinder $\MR$, by two smooth charts
$\varphi^\pm:\Ccal\to U^\pm\subset\MR$, and fix a measurable partition
$\MR=V^+\cup V^-$ with $V^\pm\subset U^\pm$.
It suffices to carry out the construction on one chart $\varphi:\Ccal\to U$,
since taking the union over the two atlas only multiplies $N$ and the
constants by $J$.

Let $m:=\lceil n^{1/d}\rceil$ and partition $\Ccal$ into $m^d\le 2^d n$ subcubes
$Q_1,\dots,Q_{m^d}$ of side $h=2/m\asymp n^{-1/d}$. Transport the reference
Lagrange nodes to each $Q_j$ as in Lemma~\ref{lem:banach-bramble-hilbert} and take
as skeleton points their chart images $\zeta_{j,i}:=\varphi(y_i^{Q_j})$, so that
after summing over the two atlas one still has $N\le Ln$. For $\theta\in V$
with $\eta:=\varphi^{-1}(\theta)\in Q_j$, let $\alpha(\theta)$ collect the
Lagrange weights $\ell_i^{Q_j}(\eta)$ at the indices $(j,i)$ and vanish elsewhere.
These are measurable and bounded by the reference Lebesgue constant.

By Lemma~\ref{lem:local-parameter-smoothness},
$P:=\phi^k_{\varphi(\cdot)}\in C^{k,1/p}(\Ccal;L^p(\Omega))$ with
\begin{align*}
    \abs{P}_{C^{k,1/p}(Q_j;L^p(\Omega))}
    \le\abs{P}_{C^{k,1/p}(\Ccal;L^p(\Omega))}\le C
\end{align*}
for every $j$. Lemma~\ref{lem:banach-bramble-hilbert} on $Q_j$, with
$\alpha=\tfrac1p$, then gives, for $\eta\in Q_j$,
\begin{align*}
    \Big\|P(\eta)-\sum_i\ell_i^{Q_j}(\eta)\,P(y_i^{Q_j})\Big\|_{L^p(\Omega)}
    \le C\,h^{k+\frac1p}
    \le C\,n^{-\frac{k+1/p}{d}}=C\,n^{-\gamma_{k,p}},
\end{align*}
which is \eqref{eq:skeleton-error}.
\end{proof}

\subsection{Sampling lemma in type-\texorpdfstring{$\min(p,2)$}{min(p,2)} spaces}
We recall the relevant Banach-space geometry. A Banach space $B$ has
\emph{Rademacher type} $r\in[1,2]$ if there is a constant $T_r(B)<\infty$ such
that for every finite sequence $x_1,\dots,x_q\in B$,
\begin{align}
    \mathbb E_\varepsilon
    \norm{\sum_{j=1}^q \varepsilon_j x_j}_B^r
    \le
    T_r(B)^r\sum_{j=1}^q\norm{x_j}_B^r,
    \label{eq:rademacher-type-def}
\end{align}
where $\varepsilon_1,\dots,\varepsilon_q$ are independent Rademacher signs. By the
standard symmetrization argument, \eqref{eq:rademacher-type-def} implies that for
independent mean-zero $B$-valued random variables $W_1,\dots,W_q$,
\begin{align}
    \mathbb E\norm{\sum_{j=1}^q W_j}_B^r
    \le
    C_r\,T_r(B)^r\sum_{j=1}^q\mathbb E\norm{W_j}_B^r .
    \label{eq:type-r-independent}
\end{align}
We shall only use this consequence below. For $1<p<\infty$, the space
$L^p(\Omega)$ has Rademacher type
\begin{align}
    r_p:=\min(p,2),
    \label{eq:lp-type-rp}
\end{align}
and this type is optimal in infinite-dimensional $L^p$ spaces. See
\cite[Chapter 6]{albiac_TopicsBanachSpace_2006}; for the application of
this geometry, see also~\cite{siegel_sharpboundsapproximation_2024}.

\begin{lemma}
\label{lem:high-prob-uniform-hilbert}
Let $1<p<\infty$, set $r_p:=\min(p,2)$ and
$\alpha_p:=1-\tfrac1{r_p}$. There is a constant $C_p>0$ with the following
property. Let $(X,\mu)$ be a probability space and
$Y\in L^{r_p}(X,\mu;L^p(\Omega))$; set
$A:=\int_X Y\,d\mu$ and
\begin{align*}
    \sigma^{r_p}:=\int_X\norm{Y-A}_{L^p(\Omega)}^{r_p}\,d\mu.
\end{align*}
Let $\delta\in(0,1)$, let $m\ge16\log(2/\delta)$, and let
$\Theta_1,\dots,\Theta_m\overset{\mathrm{iid}}\sim\mu$. Then there are
coefficients $\beta_1,\dots,\beta_m$, measurable in the samples, such that
\begin{align}
    \mathbb P\left(
    \Big\|\sum_{j=1}^m\beta_j\,Y(\Theta_j)-A\Big\|_{L^p(\Omega)}
    \le C_p\,\sigma
    \left(\frac{\log(2/\delta)}{m}\right)^{\alpha_p}
    \right)\ge1-\delta.
    \label{eq:type2-sampling}
\end{align}
For $p=2$ the space is Hilbertian and one may take $C_2$ a universal constant.
For $2\le p<\infty$ this recovers the type-$2$ estimate with
$\alpha_p=\tfrac12$; for $1<p<2$ it gives the type-$p$ rate
$\alpha_p=\tfrac1{p'}$.
\end{lemma}

\begin{proof}
Write $r:=r_p$ and $\alpha:=1-\tfrac1r$. By \eqref{eq:type-r-independent} and 
\eqref{eq:lp-type-rp}, the space $L^p(\Omega)$ satisfies the type-$r$ inequality 
for independent mean-zero random variables. Let $\tau_p := C_r^{1/r}\,T_r(L^p(\Omega))$ 
denote the corresponding constant. For $p=2$, $\tau_p$ is universal; for $p\ge2$ one may 
take $\tau_p\lesssim\sqrt p$.

Set $\ell:=\log(2/\delta)$, $B:=\lceil8\ell\rceil$, $q:=\lfloor m/B\rfloor$, and
split $\{1,\dots,m\}$ into $B$ disjoint blocks $I_1,\dots,I_B$ of size $q$,
discarding the leftover indices. Let
\begin{align*}
    Z_b:=q^{-1}\sum_{j\in I_b}Y(\Theta_j).
\end{align*}
Then $\mathbb E Z_b=A$, and \eqref{eq:type-r-independent} gives
\begin{align*}
    \mathbb E\norm{Z_b-A}_{L^p}^r
    \le \tau_p^r\,\sigma^r\,q^{1-r}.
\end{align*}
Set
\begin{align*}
    R_\ast:=4^{1/r}\tau_p\,\sigma\,q^{1/r-1}.
\end{align*}
By Markov's inequality, each block is \emph{good}, meaning
$\norm{Z_b-A}_{L^p}\le R_\ast$, with probability at least $3/4$, independently
across blocks. Hoeffding's inequality gives
\begin{align*}
    \mathbb P\big(\#\{\text{good }b\}\le B/2\big)\le e^{-B/8}\le e^{-\ell}
    =\delta/2 < \delta.
\end{align*}
Work on the complementary event, where good blocks form a strict majority. For
each $b$ let the median\footnote{the middle value of a sorted 
list of numbers.} $m_b:=\med_{1\le c\le B}\norm{Z_b-Z_c}_{L^p}$ and choose
$\widehat b$ with $m_{\widehat b}$ minimal, using a deterministic tie-breaking
rule. If $b$ is good, then $\norm{Z_b-Z_c}_{L^p}\le2R_\ast$ for every good $c$,
so $m_b\le2R_\ast$. Conversely, if $\norm{Z_b-A}_{L^p}>3R_\ast$, then
$\norm{Z_b-Z_c}_{L^p}>2R_\ast$ for every good $c$, so $m_b>2R_\ast$. Hence the
minimizer satisfies
\begin{align*}
    \norm{Z_{\widehat b}-A}_{L^p}\le3R_\ast .
\end{align*}
Finally $\ell\ge\log2$ gives $B\le10\ell$, and $m\ge16\ell$ gives
$q\ge m/(3B)\ge m/(30\ell)$, so
\begin{align*}
    3R_\ast
    \le C_p\,\sigma\left(\frac{\ell}{m}\right)^{1-1/r}
    = C_p\,\sigma\left(\frac{\log(2/\delta)}{m}\right)^{\alpha_p}.
\end{align*}
Setting $\beta_j:=q^{-1}\mathbf 1_{\{j\in I_{\widehat b}\}}$ realizes
$\sum_j\beta_jY(\Theta_j)=Z_{\widehat b}$ and proves \eqref{eq:type2-sampling}.
\end{proof}

\subsection{Proof of Theorem~\ref{thm:high-prob-random-feature}}

\begin{proof}[Proof of Theorem~\ref{thm:high-prob-random-feature}]
Take as deterministic points $\zeta_1,\dots,\zeta_N$ the skeleton points of
Proposition~\ref{prop:deterministic-skeleton}; then $N\le Ln$. Set
$\ell_\delta:=\log(2/\delta)$.

\emph{Step 1: decomposition.} By Definition~\ref{def:radon-lp-representation-space}
(and, when $p=2$, Theorem~\ref{thm:bounded-bias-decomposition} together with
Corollary~\ref{cor:quotient-norm-equiv-sobolev}), choose a near-minimizing density
$g\in L^p(\MR)$ with
\begin{align*}
    u=v_g\ \text{ in }L^p(\Omega),
    \quad
    v_g:=\int_{\MR}\phi^k_\theta\,g(\theta)\,d\lambda(\theta),
    \quad
    \norm{g}_{L^p(\MR)}\le2\,\norm{u}_{\mathcal{R}L^p_k(\Omega)}.
\end{align*}
Fix a measurable representative of $g$; all point values $g(\Theta_j)$ below refer
to it, and changing it on a $\lambda$-null set affects nothing. Since $u=v_g$, it
suffices to find coefficients with
\begin{align}
    \mathbb P\left(
    \Big\|v_g-\sum_i a_i\phi^k_{\zeta_i}-\sum_j b_j\phi^k_{\Theta_j}\Big\|_{L^p(\Omega)}
    \le C\ell_\delta^{\alpha_p}\,\norm{g}_{L^p(\MR)}\,n^{-\alpha_p-\gamma_{k,p}}
    \right)\ge1-\delta.
    \label{eq:vg-target-lp}
\end{align}

\emph{Step 2: skeleton split.} With $\alpha_i$ as in
Proposition~\ref{prop:deterministic-skeleton}, set
\begin{align*}
    \psi^k_\theta:=\phi^k_\theta-\sum_i\alpha_i(\theta)\,\phi^k_{\zeta_i},
    \quad
    r_g:=\int_{\MR}\psi^k_\theta\,g(\theta)\,d\lambda(\theta),
\end{align*}
and denote $d_g:=v_g-r_g \in\spanop\{\phi^k_{\zeta_i}\}$. Since 
$\lambda(\MR)<\infty$, H\"older's inequality and \eqref{eq:skeleton-error} give
\begin{align}
    \norm{r_g}_{L^p(\Omega)}
    \le\Lambda_R^{1-1/p}\,\norm{g}_{L^p(\MR)}\sup_{\theta\in\MR}\norm{\psi^k_\theta}_{L^p(\Omega)}
    \le C\,n^{-\gamma_{k,p}}\norm{g}_{L^p(\MR)}.
    \label{eq:rg-deterministic-lp}
\end{align}

\emph{Step 3: small $n$.} If $n<16\ell_\delta$, take $b_j\equiv0$ and approximate
$v_g$ by $d_g$ alone: since
\begin{align*}
    n^{-\gamma_{k,p}}
    \le 16^{\alpha_p}\ell_\delta^{\alpha_p}
    n^{-\alpha_p-\gamma_{k,p}}
\end{align*}
in this regime, \eqref{eq:rg-deterministic-lp} already implies
\eqref{eq:vg-target-lp}, deterministically.

\emph{Step 4: sampling the residual.} If $n\ge16\ell_\delta$, apply
Lemma~\ref{lem:high-prob-uniform-hilbert} on $(\MR,\mu_R)$ with
$Y(\theta):=\Lambda_R\,g(\theta)\,\psi_\theta$, whose mean is $r_g$. Write
$r:=\min(p,2)$. Since $r\le p$ and $\lambda(\MR)<\infty$, the finite-measure
embedding $L^p(\MR)\hookrightarrow L^r(\MR)$, together with \eqref{eq:skeleton-error},
gives
\begin{align}
    \int_{\MR}\norm{Y}_{L^p(\Omega)}^r\,d\mu_R
    &=\Lambda_R^{r-1}\int_{\MR}\abs{g(\theta)}^r
    \norm{\psi^k_\theta}_{L^p(\Omega)}^r\,d\lambda(\theta)
    \nonumber\\
    &\le C\,n^{-r\gamma_{k,p}}\norm{g}_{L^p(\MR)}^r.
    \label{eq:Y-r-moment-lp}
\end{align}
Moreover, if $\sigma^r:=\int_{\MR}\norm{Y-r_g}_{L^p(\Omega)}^r\,d\mu_R$, then
Jensen's inequality and \eqref{eq:Y-r-moment-lp} imply
\begin{align*}
    \sigma^r
    \le 2^r\int_{\MR}\norm{Y}_{L^p(\Omega)}^r\,d\mu_R
    \le C\,n^{-r\gamma_{k,p}}\norm{g}_{L^p(\MR)}^r.
\end{align*}
Lemma~\ref{lem:high-prob-uniform-hilbert} yields coefficients $\beta_j$ such that,
with probability at least $1-\delta$,
\begin{align*}
    \widehat r:=\sum_{j=1}^n\beta_j\,Y(\Theta_j)
    =\sum_{j=1}^n c_j\,\phi^k_{\Theta_j}
     -\sum_i\Big(\sum_{j=1}^n c_j\,\alpha_i(\Theta_j)\Big)\phi^k_{\zeta_i},
    \qquad c_j:=\Lambda_R\,\beta_j\,g(\Theta_j),
\end{align*}
and
\begin{align*}
    \norm{r_g-\widehat r}_{L^p(\Omega)}
    \le C\ell_\delta^{\alpha_p}\,\norm{g}_{L^p(\MR)}
    n^{-\alpha_p-\gamma_{k,p}}.
\end{align*}
The second sum lies in $\spanop\{\phi_{\zeta_i}\}$, so $d_g+\widehat r$ has the
required form, and $v_g-(d_g+\widehat r)=r_g-\widehat r$ gives
\eqref{eq:vg-target-lp}.

Combining Steps 1--4 with
$\norm{g}_{L^p(\MR)}\le2\norm{u}_{\mathcal{R}L^p_k(\Omega)}$ proves
\eqref{eq:lp-rf-bound-Bkp}.
\end{proof}

\begin{remark}
The deterministic neurons carry only the interpolation skeleton on $\MR$. The
random neurons are used solely for the residual $\psi_\theta$, whose uniform size
$\mathcal O(n^{-\gamma_{k,p}})$ is what allows the sampling step to contribute the additional
factor $n^{-\alpha_p}$.
\end{remark}

\begin{remark}
\label{rem:why-p-ge-2}
In Section~\ref{sec:radon-lp-theory}, the $L^p$
boundedness and the Seeger--Sogge--Stein regularity of the Radon transform hold on the whole range $1<p<\infty$, so the embedding
$W^{s_{k,p},p}(\Omega)\hookrightarrow\mathcal{R}L^p_k(\Omega)$ of
Theorem~\ref{the:radon-lp-sobolev-embedding} is available there as well. What
changes for $1<p<2$ is only the sampling geometry: $L^p(\Omega)$ has Rademacher
type $p$, not type $2$. Consequently the $n$ i.i.d.\ uniform neurons average the
residual integral at the type-$p$ rate $n^{-(1-1/p)}=n^{-1/p'}$ instead of
$n^{-1/2}$. This is the dimension-free rate available under the sole
$L^p$-moment control coming from $g\in L^p(\MR)$.
\end{remark}

\begin{remark}[Rate comparison]
\label{rem:lp-rate-comparison}
Two exponents have now been produced by the theory: the smoothness exponents
$r_{k,p}\le s_{k,p}$ of the sandwich of
Theorem~\ref{the:radon-lp-sobolev-embedding}, and the rate
\begin{align}
    \beta_{k,p}:=\alpha_p+\gamma_{k,p}
    =1-\frac1{\min(p,2)} + \frac{k+1/p}{d}
    \label{eq:rate-exponent}
\end{align}
of Theorem~\ref{thm:high-prob-random-feature}. The comparison rests on the
elementary identity
\begin{align*}
    \gamma_{k,p}+\frac1{p'}
    =\frac1d\left(k+d-\frac{d-1}{p}\right)
    =\begin{cases}
        r_{k,p}/d, & 1<p\le2,\\[1mm]
        s_{k,p}/d, & 2\le p<\infty,
    \end{cases}
\end{align*}
which says that a sampling step running at the type-$p$ rate $1/p'$ would land
exactly on the right endpoint of the sandwich below $p=2$ and exactly on its left
endpoint above $p=2$. The actual sampling rate is however
$\alpha_p=\min\{\tfrac12,\tfrac1{p'}\}$, so that
\begin{align}
    \beta_{k,p}
    =\frac{r_{k,p}}{d}+\frac{d-2}{d}\left(\frac12-\frac1p\right)_+
    =\frac{s_{k,p}}{d}-
    \begin{cases}
        \dfrac{2\delta_p}{d}=\dfrac{(d-1)(2-p)}{pd}, & 1<p<2,\\[2mm]
        \dfrac12-\dfrac1p, & 2\le p<\infty .
    \end{cases}
    \label{eq:rate-vs-sandwich}
\end{align}
In particular $\beta_{k,2}=r_{k,2}/d=s_{k,2}/d=s_k/d$, one has
$\beta_{k,p}=r_{k,p}/d$ for $1<p\le2$, and
$r_{k,p}/d\le\beta_{k,p}<s_{k,p}/d$ for $2<p<\infty$, with equality on the left
exactly when $d=2$.
\end{remark}

\section{Numerical experiments}
\label{sec:numerical-experiments}

We test the effect of the deterministic skeleton in a function approximation task.
On $\Omega=[-1,1]^d$ we approximate
\begin{align*}
    u_d(x)=\prod_{i=1}^d \sin\Big(\frac{\pi x_i}{2}\Big),
    \qquad d=2,\ldots,6,
\end{align*}
using $\Sigma_M^{\sigma_k}$ with $k=3$ and
$M\in\{2^m: m=4,\ldots,11\}$.  For each inner-feature set we solve the direct
least-squares problem by \texttt{scipy.linalg.lstsq} and report the relative
$L^2(\Omega)$ error on an independent quadrature rule.  The full sweep uses thirteen independent
sampling random seeds.

The random sampler draws $w\in\Sph^{d-1}$ by Gaussian normalization and draws
$b$ uniformly from $[-\sqrt d,\sqrt d]$.  The fixed sampler uses a
quasi-uniform~\cite{mao_solvinghighdimensionalpdes_2026,he_liu_tian_divergencefree_2026}
spherical grid and a deterministic bounded-bias grid.
The mixed sampler takes a prescribed fraction (we use $0.5$ here)
of features from the fixed grid and fills the rest at random.
Hyperplanes not intersecting $\Omega$ are
filtered out in the implementation.

\begin{figure}[t]
    \centering
    \includegraphics[width=0.98\textwidth]{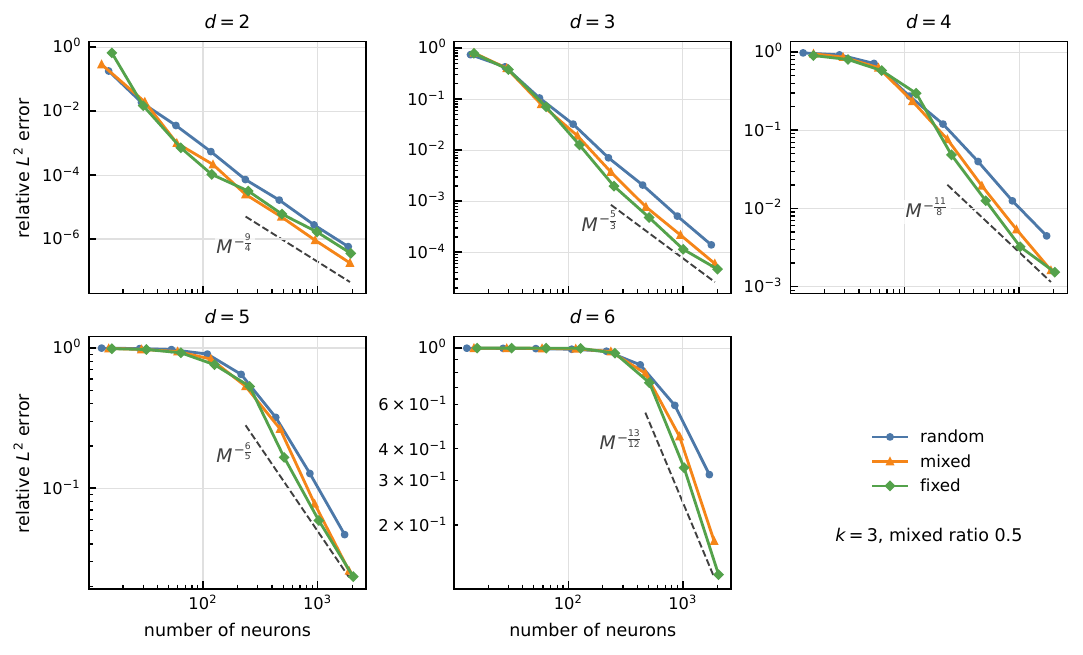}
    \caption{Convergence curves for $k=3$ and mixed ratio $0.5$.  For each requested neuron size,
    the curve is plotted at the median post-filter number of active neurons, and the vertical value
    is the median relative error over thirteen seeds. Dashed black segments show the theoretical
    slope $M^{-s_k/d}$, $s_k=(d+2k+1)/2$.}
    \label{fig:radon-rf-full-loss-curves}
\end{figure}

Figure~\ref{fig:radon-rf-full-loss-curves} gives the most direct picture.  The fully random sampler is consistently the least accurate curve.  The fixed grid is the most
accurate or nearly so, and the mixed sampler stays between these two baselines in four of the five
dimensions while still keeping the random component.  Thus the deterministic part is not only a
theoretical device: it regularizes the random projection in a way visible at the level of the whole
convergence curve.

\section{Conclusion and outlook}
\label{sec:conclusion}

We have developed a Radon-domain $L^p$ theory for shallow $\mathrm{ReLU}^k$
networks on a bounded Lipschitz domain, by studying the analysis properties of
the analysis and synthesis operators from the perspective of microlocal analysis.  
The central object is the space
$\mathcal{R}L^p_k(\Omega)=\mathrm{S}^k_R\bigl(L^p(\mathbb{S}^{d-1}\times[-R,R])\bigr)$. 
At $p=2$ elementary Fourier analysis on the Radon
side identifies it with the critical Sobolev space $H^{s_k}(\Omega)$,
$s_k=(k+1)+\tfrac{d-1}{2}$, the two summands indicating the $k+1$ orders carried
by the activation and the $\tfrac{d-1}{2}$ orders supplied by the dual Radon
transform.  For $1<p<\infty$, the Hilbertian identity degrades into the sharp
Sobolev sandwich
$W^{s_{k,p},p}(\Omega)\hookrightarrow\mathcal{R}L^p_k(\Omega)\hookrightarrow
W^{r_{k,p},p}(\Omega)$, whose width $s_{k,p}-r_{k,p}=2\delta_p$ is exactly the
Seeger--Sogge--Stein loss traversed once by the analysis map $\mathrm{A}^k$ and
once by the synthesis map $\mathrm{S}^k_R$, and neither endpoint can be moved.
Discretizing the ridge integral by a deterministic interpolation skeleton
together with $n$ i.i.d.\ uniform neurons then yields, with high probability,
the $L^p(\Omega)$ rate $n^{-\alpha_p-\gamma_{k,p}}$ from $\mathcal O(n)$
neurons; the exponent is the optimal $s_k/d$ at $p=2$, equals the benchmark
exponent $r_{k,p}/d$ of the right endpoint for $1<p\le2$, and lies strictly
between $r_{k,p}/d$ and $s_{k,p}/d$ for $p>2$.  The numerical experiments of
Section~\ref{sec:numerical-experiments} confirm that the deterministic skeleton
is not merely a proof device: the mixed sampler improves on the purely random
one across dimensions.

Many new problems arise, both theoretical and practical. For $p\neq2$ the mixed
sampler attains $n^{-\alpha_p-\gamma_{k,p}}$ for each fixed target in $\mathcal{R}L^p_k(\Omega)$,
with $\alpha_p+\gamma_{k,p}=r_{k,p}/d$ for $1<p\le2$ and
$r_{k,p}/d\le\alpha_p+\gamma_{k,p}<s_{k,p}/d$ for $p>2$. Closing the gap requires
improving the sampling exponent, or establishing a matching lower bound through 
the still-unknown metric entropy of the unit ball of $\mathcal{R}L^p_k(\Omega)$.  
More fundamentally, the gap $2\delta_p$ is
intrinsic to the Bessel-potential scale: on spaces adapted to Fourier integral
operators the Seeger--Sogge--Stein loss disappears, suggesting an exact
characterization of $\mathcal{R}L^p_k(\Omega)$ by a single norm equivalence for
every $1<p<\infty$.  Extending the theory to the endpoints $p=1$ and
$p=\infty$ is another open problem.  Finally, our hybrid construction combines
a deterministic skeleton with uniformly sampled neurons. It remains to
determine whether a fully
deterministic skeleton can attain the same exponent, thereby yielding a linear
approximation scheme.

\newpage
\appendix
\section{Postponed proofs}\label{app:proofs}

\subsection{Proof of the analysis multiplier lemma}
\label{app:proof-of-analysis-wellposed}
We now prove Lemma~\ref{lem:analysis-multiplier-wellposed}.
\begin{proof}
For measurable $F\ge0$ on $\mathbb R^d$ the function $(w,\omega)\mapsto F(\omega w)$ is
invariant under $(w,\omega)\mapsto(-w,-\omega)$, so \eqref{eq:antipodal} reads
\begin{align}
    \int_{\mathbb S^{d-1}}\!\!\int_{\mathbb R}|\omega|^{d-1}F(\omega w)\,d\omega\,d\sigma(w)
    =2\int_{\mathbb R^d}F(\xi)\,d\xi .
    \label{eq:polar-pullback}
\end{align}

\emph{Well-definedness.} Let $E\subset\mathbb R^d$ be a Lebesgue null set. Taking
$F=\mathbf 1_E$ in \eqref{eq:polar-pullback} makes the right-hand side vanish, and
$|\omega|^{d-1}>0$ for a.e.\ $\omega$, so $\{(w,\omega):\omega w\in E\}$ is null in
$\mathbb S^{d-1}\times\mathbb R$. Hence $\widehat u(\omega w)$, and with it $m_u$, is
defined a.e.\ and is unchanged when $\widehat u$ is modified on a null set.

\emph{Temperedness.} Set $N':=N+k+1$ and apply \eqref{eq:polar-pullback} to
$F(\xi)=(1+|\xi|)^{-N'}|\xi|^{k+1}|\widehat u(\xi)|$. Since $|\omega w|=|\omega|$ and
$|\xi|^{k+1}\le(1+|\xi|)^{k+1}$,
\begin{align}
    \int_{\mathbb S^{d-1}}\!\!\int_{\mathbb R}
    \frac{|m_u(w,\omega)|}{(1+|\omega|)^{N'}}\,d\omega\,d\sigma(w)
    =2\int_{\mathbb R^d}\frac{|\xi|^{k+1}\,|\widehat u(\xi)|}{(1+|\xi|)^{N'}}\,d\xi
    \le 2\int_{\mathbb R^d}\frac{|\widehat u(\xi)|}{(1+|\xi|)^{N}}\,d\xi<\infty .
    \label{eq:mu-weighted-L1}
\end{align}
In particular $m_u\in L^1_{\mathrm{loc}}(\mathbb S^{d-1}\times\mathbb R)$, and for every
$\Phi\in\mathcal S(\mathbb S^{d-1}\times\mathbb R)$,
\begin{align*}
    \left|\int_{\mathbb S^{d-1}}\!\!\int_{\mathbb R}m_u\Phi\,d\omega\,d\sigma\right|
    \le\Bigl(\sup_{\mathbb S^{d-1}\times\mathbb R}(1+|\omega|)^{N'}|\Phi|\Bigr)
    \int_{\mathbb S^{d-1}}\!\!\int_{\mathbb R}
    \frac{|m_u|}{(1+|\omega|)^{N'}}\,d\omega\,d\sigma ,
\end{align*}
so $m_u\in\mathcal S'(\mathbb S^{d-1}\times\mathbb R)$. As $\mathcal F_b$ is a bijection of
$\mathcal S'(\mathbb S^{d-1}\times\mathbb R)$, the equation $\mathcal F_bg=c_d\,m_u$ has the
unique solution $g=\mathcal F_b^{-1}(c_d\,m_u)$.

\emph{The admissible class.} For $u\in H^{s}(\mathbb R^d)$ and any $N\ge0$ with $N>d/2-s$,
Cauchy--Schwarz gives
\begin{align*}
    \int_{\mathbb R^d}\frac{|\widehat u(\xi)|}{(1+|\xi|)^{N}}\,d\xi
    \le\bigl\|(1+|\xi|)^{s}\widehat u\bigr\|_{L^2(\mathbb R^d)}
      \bigl\|(1+|\xi|)^{-N-s}\bigr\|_{L^2(\mathbb R^d)}<\infty .
\end{align*}
For compactly supported $u\in\mathcal S'(\mathbb R^d)$, the Paley--Wiener--Schwartz
theorem~\cite{friedlander_IntroductionTheoryDistributions_2003} gives
$\widehat u\in C^\infty(\mathbb R^d)$ with $|\widehat u(\xi)|\le C(1+|\xi|)^{M}$ for some
$M\ge0$, so any $N>M+d$ is admissible.
\end{proof}

\subsection{Proof of the existence of \texorpdfstring{$k$}{k}-moments
of \texorpdfstring{$g_u$}{g_u}}\label{app:proof-of-k-moments}
We now prove Lemma~\ref{lem:canonical-density-weighted-L1},
\begin{proof}
Write
\begin{align*}
    \widehat g_u(w,\omega):=\mathcal F_b g_u(w,\omega),
    \qquad
    \Phi(\omega):=(i\omega)^{k+1}|\omega|^{d-1}.
\end{align*}
We first show that
\begin{align*}
    \partial_\omega^j \widehat g_u\in L^2(\Sph^{d-1}\times\R),
    \qquad 0\le j\le k+1.
\end{align*}
For $0\le r\le k+1$, the distributional derivative
$\partial_\omega^r\Phi$ is represented by a locally integrable function and satisfies
\begin{align*}
    |\partial_\omega^r\Phi(\omega)|
    \le C_r |\omega|^{k+d-r}
    \qquad\text{for a.e. }\omega\in\R .
\end{align*}
Indeed, $\Phi$ is a piecewise polynomial homogeneous function of degree $k+d$, and
no singular term at $\omega=0$ appears for derivatives of order $r\le k+1\le k+d$.
Using the Leibniz rule in the distributional sense, we obtain, for $0\le j\le k+1$,
\begin{align*}
    |\partial_\omega^j \widehat g_u(w,\omega)|^2
    \le C
    \sum_{\ell=0}^j
    |\omega|^{2(k+d-j+\ell)}
    |(w\cdot\nabla)^\ell \widehat u(\omega w)|^2 .
\end{align*}
For each term in the sum, the polar-coordinate identity gives
\begin{align*}
    &\int_{\Sph^{d-1}}\int_\R
    |\omega|^{2(k+d-j+\ell)}
    |(w\cdot\nabla)^\ell \widehat u(\omega w)|^2
    \,d\omega\,d\sigma(w)  \nonumber\\
    &\qquad\le
    C\int_{\R^d}
    |\xi|^{2s_k-2(j-\ell)}
    |\nabla^\ell \widehat u(\xi)|^2\,d\xi ,
\end{align*}
where $2s_k=2k+d+1$. Since $0\le j-\ell\le j\le k+1$, we have
\begin{align*}
    0\le 2s_k-2(j-\ell)\le 2s_k .
\end{align*}
Hence $|\xi|^{2s_k-2(j-\ell)} \le (1+|\xi|^2)^{s_k}$, and therefore
\begin{align*}
    \int_{\R^d}
    |\xi|^{2s_k-2(j-\ell)}
    |\nabla^\ell \widehat u(\xi)|^2\,d\xi
    \le
    C\sum_{|\alpha|=\ell}
    \|x^\alpha u\|_{H^{s_k}(\R^d)}^2 .
\end{align*}
Because $u$ is compactly supported, multiplication by $x^\alpha$ preserves
$H^{s_k}$: choosing $\chi\in C_c^\infty(\R^d)$ with $\chi=1$ near
$\operatorname{supp}u$, one has $x^\alpha u=(x^\alpha\chi)u\in H^{s_k}(\R^d)$.
Thus $\partial_\omega^j\widehat g_u\in L^2(\Sph^{d-1}\times\R)$ for all
$0\le j\le k+1$. By Plancherel in the $b$-variable,
\begin{align*}
    b^j g_u\in L^2(\Sph^{d-1}\times\R),
    \qquad 0\le j\le k+1.
\end{align*}
Finally, by Cauchy--Schwarz,
\begin{align*}
    \|(1+|b|)^k g_u\|_{L^1(\Sph^{d-1}\times\R)}
    &\le
    \left(
        \abs{\Sph^{d-1}}
        \int_\R (1+|b|)^{-2}\,db
    \right)^{1/2}
    \|(1+|b|)^{k+1}g_u\|_{L^2}  \\
    &\le
    C\sum_{j=0}^{k+1}
    \|b^j g_u\|_{L^2(\Sph^{d-1}\times\R)}
    <\infty .
\end{align*}
\end{proof}

\subsection{Proof of polynomial realization Lemma~\ref{lem:polynomial-lp-density}}
\label{app:proof-poly-realization}

\begin{proof}
The construction is algebraic and produces a single density $g_q$ whose only
$p$-dependence is in the final norm bound, where equivalence of norms on a
finite-dimensional space is invoked.

For $b\in[-R,-R_0]$ and $x\in\Omega$ one has $w\cdot x-b\ge R_0-R_0=0$, so
$\sigma_k(w\cdot x-b)=(w\cdot x-b)^k/k!$. Hence, for any $g$ supported in
$\Sph^{d-1}\times[-R,-R_0]$, expanding the binomial gives, for $x\in\Omega$,
\begin{align}
    \mathrm{S}^k_R g(x)
    &=\sum_{i=0}^k\int_{\Sph^{d-1}}a_i(w)\,(w\cdot x)^i\,d\sigma(w),
    \\
    a_i(w):&=\frac{\binom{k}{i}}{k!}\int_{-R}^{-R_0}(-b)^{k-i}g(w,b)\,db.
    \label{eq:AR-poly-expansion}
\end{align}

\emph{Matching the homogeneous parts.} Write $q=\sum_{i=0}^k q^{(i)}$ with
$q^{(i)}=\sum_{\abs\alpha=i}c_\alpha x^\alpha$ homogeneous of degree $i$. A
homogeneous polynomial vanishing on $\Sph^{d-1}$ vanishes identically, so the
monomials $\{w^\alpha:\abs\alpha=i\}$ are linearly independent in
$L^2(\Sph^{d-1})$ and their Gram matrix is invertible. Choosing $a_i$ in their
span with
$\binom{i}{\alpha}\int_{\Sph^{d-1}}a_i(w)\,w^\alpha\,d\sigma(w)=c_\alpha$ for all
$\abs\alpha=i$ makes the $i$-th term of \eqref{eq:AR-poly-expansion} equal to
$q^{(i)}$; here $\binom{i}{\alpha}=i!/\alpha!$. Summing over $i$ yields
$\sum_i\int_{\Sph^{d-1}}a_i(w)(w\cdot x)^i\,d\sigma=q(x)$ on $\Omega$, and since
$q\mapsto(a_i)$ is linear on the finite-dimensional space $\mathcal P_k(\R^d)$,
$\sum_i\norm{a_i}_{L^\infty(\Sph^{d-1})}\le C\norm{q}_{L^p(\Omega)}$.

\emph{Realizing the $b$-moments.} The monomials $\{(-b)^m\}_{m=0}^k$ are linearly
independent on $[-R,-R_0]$, so there is a bounded linear map
$\R^{k+1}\to\operatorname{span}\{(-b)^m\}\subset L^\infty([-R,-R_0])$ sending
$(\nu_0,\dots,\nu_k)$ to a function $\beta$ with
$\int_{-R}^{-R_0}(-b)^m\beta(b)\,db=\nu_m$, $0\le m\le k$. Applying it pointwise
in $w$ with $\nu_{k-i}(w):=\tfrac{k!}{\binom{k}{i}}a_i(w)$ defines
$g_q(w,\cdot):=\beta$, which realizes the moments in \eqref{eq:AR-poly-expansion}
and hence $\mathrm{S}^k_R g_q=q$ on $\Omega$. Finally, since
$\Sph^{d-1}\times[-R,-R_0]$ has finite measure and $g_q$ is built from $a_i$ by a
fixed bounded linear map,
\begin{align*}
    \norm{g_q}_{L^p(\MR)}
    \le C\sum_{m=0}^k\norm{\nu_m}_{L^\infty(\Sph^{d-1})}
    \le C\sum_{i=0}^k\norm{a_i}_{L^\infty(\Sph^{d-1})}
    \le C\norm{q}_{L^p(\Omega)} ,
\end{align*}
where the last step uses equivalence of all norms on the finite-dimensional space
$\mathcal P_k(\R^d)$ (here $\Omega$ has nonempty interior, so $\norm{\cdot}_{L^p(\Omega)}$
is a norm on $\mathcal P_k(\R^d)$).
\end{proof}

\section{Background on Fourier integral operators}
\label{app:fio-background}
This appendix recalls the microlocal background needed for
Lemma~\ref{lem:radon-fio}, Lemma~\ref{lem:inverse-fio-pair}, and the sharpness
argument in Proposition~\ref{prop:radon-lp-sandwich-sharp}: oscillatory integrals,
the H\"ormander order of a Fourier integral operator (FIO), the wave front set,
Lagrangian submanifolds, canonical transformations, composition, and the natural
projection rank condition in the Seeger--Sogge--Stein sharpness theorem.
Some proofs are omitted and replaced by references, chiefly to
\cite{sogge_FourierIntegralsClassical_2017,%
hormander_LagrangianDistributionsFourier_2009,%
grigis_MicrolocalAnalysisDifferential_1994}.
We assume only familiarity with smooth manifolds~\cite{warner_FoundationsDifferentiableManifolds_1983}
and the distribution theory~\cite{friedlander_IntroductionTheoryDistributions_2003}.
Throughout, $X$ and $Y$ are smooth
boundaryless manifolds with $n_X:=\dim X$, $n_Y:=\dim Y$; $T^*X\setminus0$ is the
cotangent bundle with the zero section removed; and \emph{conic} means invariant
under the dilations $(x,\xi)\mapsto(x,t\xi)$, $t>0$.  In the running example
$X=\R^d$ and $Y=\Sph^{d-1}\times\R$, so $n_X=n_Y=d$.

\subsection{Oscillatory integrals and symbol classes}
\label{app:osc}

\paragraph{Symbols.}
For $m\in\R$, the symbol class $S^m(X\times\R^N)$ consists of
$a\in C^\infty(X\times\R^N)$ such that, for every compact $K\subset X$ and all
multi-indices $\alpha,\gamma$,
\begin{align*}
    \abs{\partial_x^{\gamma}\partial_\theta^\alpha a(x,\theta)}
    \le C_{K,\alpha,\gamma}\,(1+\abs\theta)^{\,m-\abs\alpha},
    \qquad x\in K,\ \theta\in\R^N.
\end{align*}
A symbol is \emph{classical} (polyhomogeneous) if it
admits an asymptotic expansion $a\sim\sum_{j\ge0}a_{m-j}$ with $a_{m-j}$ positively
homogeneous of degree $m-j$ in $\theta$ for $\abs\theta\ge1$; its \emph{principal
symbol} is $a_m$, and $a$ is \emph{elliptic} if $a_m(x,\theta)\neq0$ for
$\theta\neq0$.

\paragraph{Phase functions.}
A \emph{phase function} on $X$ with $N$ phase variables is a real-valued
$\varphi\in C^\infty\bigl(X\times(\mathbb{R}^N\setminus0)\bigr)$, positively
homogeneous of degree $1$ in $\theta$, with $d_{x,\theta}\varphi\neq0$ throughout
$X\times(\mathbb{R}^N\setminus0)$.  Its \emph{critical set} is
\begin{align}
    C_\varphi:=\bigl\{(x,\theta)\in X\times(\mathbb{R}^N\setminus0):\
    d_\theta\varphi(x,\theta)=0\bigr\},
    \label{eq:critical-set}
\end{align}
and $\varphi$ is \emph{nondegenerate} if the $N$ differentials
$d_{x,\theta}\bigl(\partial_{\theta_1}\varphi\bigr),\dots,
d_{x,\theta}\bigl(\partial_{\theta_N}\varphi\bigr)$ are linearly independent at
every point of $C_\varphi$.  The first condition is what makes \eqref{eq:osc-reg}
below meaningful; the second makes $C_\varphi$ a conic submanifold of codimension
$N$ and produces the Lagrangian of Section~\ref{app:lagrangian}.

\paragraph{Oscillatory integrals.}
Let $\varphi$ be a phase function and $a\in S^m(X\times\mathbb{R}^N)$.  For
$m<-N$ the integral
\begin{align}
    I_\varphi(a)(x)
    :=(2\pi)^{-N}\int_{\mathbb{R}^N}e^{\,i\varphi(x,\theta)}\,a(x,\theta)\,d\theta
    \label{eq:osc-int}
\end{align}
converges absolutely.  For general $m$ it is defined as a distribution by
regularization: fixing $\rho\in C_c^\infty(\mathbb{R}^N)$ with $\rho(0)=1$,
\begin{align}
    \bigl\langle I_\varphi(a),u\bigr\rangle
    :=\lim_{\varepsilon\downarrow0}(2\pi)^{-N}\iint
    e^{\,i\varphi(x,\theta)}\,a(x,\theta)\,\rho(\varepsilon\theta)\,u(x)
    \,d\theta\,dx,
    \qquad u\in C_c^\infty(X),
    \label{eq:osc-reg}
\end{align}
the limit existing in $\mathcal D'(X)$ and being independent of $\rho$.  The
convergence comes from repeated integration by parts with the first-order operator
\begin{align*}
    L=\frac{1}{i\Phi}\Bigl(
    |\theta|^{2}\sum_{j=1}^{N}(\partial_{\theta_j}\varphi)\,\partial_{\theta_j}
    +\sum_{k=1}^{n_X}(\partial_{x_k}\varphi)\,\partial_{x_k}\Bigr),
    \qquad
    \Phi:=|\theta|^{2}|d_\theta\varphi|^{2}+|d_x\varphi|^{2},
\end{align*}
which is nonsingular because $d_{x,\theta}\varphi\neq0$, satisfies
$L\,e^{i\varphi}=e^{i\varphi}$, and whose transpose lowers the amplitude order by
one at each application, at the cost of one $x$-derivative falling on $u$
\cite[Ch.~1]{grigis_MicrolocalAnalysisDifferential_1994}.

\subsection{Fourier integral operators and their order}
\label{app:order}

A \emph{Fourier integral operator} $T:C_c^\infty(X)\to\mathcal D'(Y)$ is an
operator whose Schwartz kernel is locally an oscillatory integral
\begin{align}
    K_T(y,x)=(2\pi)^{-N}\int_{\R^N}e^{\,i\varphi(y,x,\theta)}\,a(y,x,\theta)\,d\theta,
    \qquad a\in S^m,
    \label{eq:fio-kernel}
\end{align}
with $\varphi$ a nondegenerate phase on $Y\times X$ in the sense
of~\eqref{eq:critical-set}.  The operator is \emph{properly supported} if the two
coordinate projections
$\operatorname{supp}K_T\to Y, \, \operatorname{supp}K_T\to X$
are proper maps, which means preimages of compact sets remain compact.
Since all estimates used in this paper are
local, one may insert compact cutoffs or use the standard proper-support
decomposition and work microlocally with properly supported pieces. See
\cite{sogge_FourierIntegralsClassical_2017,
hormander_LagrangianDistributionsFourier_2009}.

The \emph{H\"ormander order} of \eqref{eq:fio-kernel} is
\cite{hormander_LagrangianDistributionsFourier_2009}
\begin{align}
    \mu \,=\, m\,+\,\frac N2\,-\,\frac{n_X+n_Y}{4},
    \label{eq:fio-order}
\end{align}
and one writes $T\in I^\mu(Y,X;\mathcal{C})$, where $\mathcal{C}$ is the canonical relation of
Section~\ref{app:lagrangian}.  The normalization~\eqref{eq:fio-order} is fixed by the
requirement that a pseudodifferential operator of order $m$ have FIO order $m$:
its kernel $(2\pi)^{-n}\!\int e^{i(x-y)\cdot\xi}a(x,\xi)\,d\xi$ on $X=Y=\R^n$ has
$N=n$ phase variables and amplitude order $m$, and \eqref{eq:fio-order} returns
$m+\tfrac n2-\tfrac{2n}{4}=m$.

\paragraph{Pseudodifferential operators.}
The pseudodifferential operators are exactly the FIOs carrying the phase
$(x-y)\cdot\xi$.  We write $\Psi^m(X)$ for those of order $m$: $P\in\Psi^m(X)$ acts
locally by
\begin{align*}
    Pf(x)=(2\pi)^{-n_X}\int_{\R^{n_X}}\int_{\R^{n_X}}
    e^{\,i(x-y)\cdot\xi}\,p(x,\xi)\,f(y)\,dy\,d\xi,
    \qquad p\in S^m.
\end{align*}
Their canonical relation is the diagonal of
$(T^*X\setminus0)\times(T^*X\setminus0)$, the graph of the identity, and unlike a
generic FIO they map $W^{s,p}(X)\to W^{s-m,p}(X)$ for every $1<p<\infty$ with
\emph{no} extra loss (the Calder\'on--Zygmund theory). The loss $\delta_p$ of
Theorem~\ref{thm:sss-fio-lp} is the price of a canonical relation that genuinely
moves the base point, which the diagonal does not.

\paragraph{The Radon transform.}
The Radon transform $\mathcal R:C_c^\infty(\R^d)\to\mathcal D'(Y)$,
$Y=\Sph^{d-1}\times\R$, has Schwartz kernel $\delta(b-w\cdot x)$, i.e. the
oscillatory representation
\begin{align}
    \mathcal Rf(w,b)
    =(2\pi)^{-1}\int_{\R}\int_{\R^d}e^{\,i\tau(b-w\cdot x)}f(x)\,dx\,d\tau,
    \label{eq:radon-oscillatory}
\end{align}
with a single phase variable $\tau\in\R\setminus0$ (so $N=1$), phase
$\varphi=\tau(b-w\cdot x)$, and amplitude $\equiv1$ (so $m=0$).  The phase is
nondegenerate: $\partial_\tau\varphi=b-w\cdot x$, whose differential has
$db$-component $1$ and is therefore nonzero on $C_\varphi=\{b=w\cdot x\}$.  By
\eqref{eq:fio-order} with $n_X=n_Y=d$,
\begin{align*}
    \mu_{\mathcal R}=0+\tfrac12-\tfrac{d+d}{4}=-\tfrac{d-1}{2}.
\end{align*}
The amplitude $\equiv1$ never vanishes, so $\mathcal R$ is elliptic.  Its canonical
relation is identified in Section~\ref{app:lagrangian}, and the order, ellipticity, and
local-graph property are assembled into a proof of Lemma~\ref{lem:radon-fio} in
Section~\ref{app:radon-fio-proof}.

\subsection{The wave front set}
\label{app:wf}

\paragraph{Heuristic.}
The Fourier transform turns smoothness into decay: a compactly supported
distribution is smooth exactly when its Fourier transform is rapidly decreasing.
Localizing by a cutoff $\phi$ supported near a point $x_0$, the piece $\phi u$ is
smooth near $x_0$ if and only if $\widehat{\phi u}(\xi)$ decays rapidly in
every direction.  A singularity at $x_0$ persists only along those frequency
directions $\xi$ in which $\widehat{\phi u}$ fails to rapidly decay.  The wave front
set reflects precisely this directional refinement of the singular support: not
merely where $u$ is singular, but the non-smooth directions. 
These directions are intrinsically covectors: the
frequency $\xi$ enters only through the phase $x\cdot\xi$, where it is paired with
the displacement $x$, so under a change of coordinates $x=\kappa(\tilde x)$ it
transforms by the inverse transpose ${}^{t}(d\kappa)^{-1}$ of the Jacobian, which is the
transformation law of a cotangent vector, the one that keeps $\xi\,dx$ invariant.
The set of non-decay directions at $x_0$ is thus a conic subset of the fibre
$T^*_{x_0}X$, and $WF(u)$ lives in the cotangent bundle $T^*X\setminus0$ rather than
in $TX$.

\paragraph{Definition.}
For $u\in\mathcal D'(X)$, a point $(x_0,\xi_0)\in T^*X\setminus0$ is
\emph{not} in the wave front set $WF(u)$ if there exist $\phi\in C_c^\infty(X)$
with $\phi(x_0)\neq0$ and an open cone $\Gamma\ni\xi_0$ such that, in a coordinate
chart,
\begin{align*}
    \abs{\widehat{\phi u}(\xi)}\le C_M\,(1+\abs\xi)^{-M}
    \qquad\text{for all }\xi\in\Gamma\text{ and all }M\in\N.
\end{align*}
Thus $WF(u)\subset T^*X\setminus0$ is closed, conic, and coordinate-free.  Let
$\pi_X:T^*X\to X$, $(x,\xi)\mapsto x$, denote the base (bundle) projection; then
\begin{align*}
    \pi_X\bigl(WF(u)\bigr)=\operatorname{sing\,supp}u,
\end{align*}
so the wave front set refines the singular support.  Pseudodifferential operators
are \emph{microlocal}: $WF(Pu)\subset WF(u)$ for every $P\in\Psi^m(X)$.

\paragraph{Examples.}
The point mass has full cotangent fibre over its support,
$WF(\delta_{x_0})=\{x_0\}\times(\R^{n_X}\setminus0)$.  More generally, if
$S\subset X$ is a smooth submanifold and $u$ is \emph{conormal} to $S$ (e.g.\ the
surface delta $\delta_S$), then
\begin{align}
    WF(u)\subset N^*S\setminus0,
    \qquad
    N^*S:=\bigl\{(x,\xi)\in T^*X:\ x\in S,\ \xi|_{T_xS}=0\bigr\},
    \label{eq:conormal-wf}
\end{align}
the conormal bundle of $S$.  In particular the Radon kernel $\delta(b-w\cdot x)$
is conormal to the incidence manifold $Z=\{(w,b,x):b=w\cdot x\}\subset Y\times X$,
so $WF\bigl(K_{\mathcal R}\bigr)\subset N^*Z\setminus0$; after the sign twist of
Section~\ref{app:lagrangian} this conormal bundle is exactly the canonical relation
$\mathcal{C}_{\mathcal R}$ of $\mathcal R$ computed in \eqref{eq:radon-canonical-relation}
below.

\paragraph{Singularities under the Radon transform.}
The geometric meaning of $\mathcal R$ being an FIO with canonical relation
$\mathcal{C}_{\mathcal R}$ is that it transports wave front sets along $\mathcal{C}_{\mathcal R}$:
\begin{align}
    WF(\mathcal Rf)\subset \mathcal{C}_{\mathcal R}\circ WF(f).
    \label{eq:radon-wf-propagation}
\end{align}
Concretely, a singularity of $f$ at $(x_0,\xi_0)$ produces a singularity of
$\mathcal Rf$ over precisely the hyperplane through $x_0$ whose unit normal is
$w=\pm\,\xi_0/\abs{\xi_0}$, at bias $b=w\cdot x_0$.  The intuition is transversal
averaging: integrating $f$ over a moving family of hyperplanes detects a jump of
$f$ across a hypersurface only when the integration plane is \emph{tangent} to that
hypersurface (its normal aligned with $\xi_0$). Planes meeting the singularity
transversally average it away and leave $\mathcal Rf$ smooth there.  Estimate
\eqref{eq:radon-wf-propagation} is the qualitative shadow of the quantitative
$L^p$--Sobolev bound the SSS theorem provides.

\subsection{Lagrangian submanifolds and canonical relations}
\label{app:lagrangian}

\paragraph{Lagrangians.}
Equip $T^*X$ with the tautological one-form $\alpha_X=\xi\,dx=\sum_j\xi_j\,dx_j$
and the canonical symplectic form $\omega_X=d\alpha_X=\sum_j d\xi_j\wedge dx_j$.  A
submanifold $\Lambda\subset T^*X$ is \emph{Lagrangian} if $\dim\Lambda=n_X$ and
$\omega_X|_\Lambda=0$.  The Lagrangians relevant to FIOs are conic.  Two model
families: the conormal bundles $N^*S$ of \eqref{eq:conormal-wf}, and the
phase-generated Lagrangians
\begin{align*}
    \Lambda_\varphi
    =\bigl\{(x,d_x\varphi(x,\theta)):\ (x,\theta)\in C_\varphi\bigr\},
\end{align*}
which is a conic Lagrangian whenever $\varphi$ is nondegenerate; conversely, every
conic Lagrangian is locally of this form
\cite[\S3.1]{hormander_LagrangianDistributionsFourier_2009}.  A distribution
$u\in I^m(X,\Lambda)$ associated with $\Lambda$ satisfies $WF(u)\subset\Lambda$,
and FIO kernels are exactly the Lagrangian distributions whose Lagrangian is the
twist of a canonical relation.

\paragraph{Canonical relations, transformations, and the twist.}
On the product $T^*Y\times T^*X$, let $\operatorname{pr}_Y$ and $\operatorname{pr}_X$
denote the two factor projections, and form the \emph{twisted} symplectic form
\begin{align}
    \omega_Y\ominus\omega_X
    :=\operatorname{pr}_Y^*\omega_Y-\operatorname{pr}_X^*\omega_X .
    \label{eq:twisted-form}
\end{align}
A \emph{canonical relation} is a conic Lagrangian
$\mathcal{C}\subset(T^*Y\setminus0)\times(T^*X\setminus0)$ for $\omega_Y\ominus\omega_X$.  A
\emph{homogeneous canonical transformation} is a homogeneous diffeomorphism
$\chi:T^*X\setminus0\to T^*Y\setminus0$ with $\chi^*\omega_Y=\omega_X$.  The two
notions stand to each other as \emph{map} to \emph{graph}: the graph
$\mathcal{C}_\chi=\{(\chi(x,\xi);\,x,\xi):(x,\xi)\in T^*X\setminus0\}$ of a canonical
transformation is a canonical relation, and the identity $\chi^*\omega_Y=\omega_X$
is exactly what makes $\mathcal{C}_\chi$ Lagrangian for the \emph{minus} sign
in~\eqref{eq:twisted-form}.  Canonical relations are strictly more general: 
a relation need not be a graph, since its projections to the two factors
may degenerate. It is locally a graph and then the graph of a canonical
transformation precisely under the condition isolated below.

\emph{The twist.}  The minus sign in~\eqref{eq:twisted-form} is what couples the
two factors symplectically, but it prevents $\mathcal{C}$ from being a Lagrangian submanifold
in the usual (untwisted) sense.  The fibrewise antipode
$\iota:(x,\xi)\mapsto(x,-\xi)$ on the $X$-factor repairs this: since
$\iota^*\omega_X=-\omega_X$, applying $\mathrm{id}_{T^*Y}\times\iota$ converts a
Lagrangian for $\omega_Y\ominus\omega_X$ into an honest Lagrangian for the
\emph{sum} $\omega_{Y\times X}=\omega_Y\oplus\omega_X$, the canonical symplectic
form of $T^*(Y\times X)$.  Accordingly the kernel of $T\in I^\mu(Y,X;\mathcal{C})$ is a
Lagrangian distribution on $Y\times X$ associated with the untwisted Lagrangian
\begin{align*}
    \mathcal{C}'=\bigl\{(y,\eta;\,x,-\xi):(y,\eta;\,x,\xi)\in \mathcal{C}\bigr\}\subset T^*(Y\times X),
\end{align*}
and the order $\mu$ of $T$ is, by definition, the order of this Lagrangian
distribution.

\paragraph{The local-graph condition.}
The two projections
\begin{align*}
    \pi_L:\mathcal{C}\to T^*Y\setminus0,
    \qquad
    \pi_R:\mathcal{C}\to T^*X\setminus0
\end{align*}
need not be diffeomorphisms in general.  One says $\mathcal{C}$ is a \emph{local canonical
graph} when both $\pi_L$ and $\pi_R$ are local diffeomorphisms, equivalently when
$\mathcal{C}$ is locally the graph $\mathcal{C}_\chi$ of a canonical transformation $\chi$.  This
nondegeneracy is precisely the geometric hypothesis under which the
Seeger--Sogge--Stein $L^p$--Sobolev regularity theorem
(Theorem~\ref{thm:sss-fio-lp}) applies, with the sharp loss
$\delta_p=(d-1)\abs{1/p-1/2}$.

\subsection{Proof of the Radon FIO lemma}
\label{app:radon-fio-proof}

We can now collect the verification deferred from the main text and prove
Lemma~\ref{lem:radon-fio}.

\begin{proof}[Proof of Lemma~\ref{lem:radon-fio}]
We argue for $\mathcal R$; the dual $\mathcal R^*$ follows by transposition.
Throughout we use the standard FIO calculus for smooth densities: passing from
half-densities to the product measures $d\sigma(w)\,db$ on $Y$ and $dx$ on $\R^d$
conjugates the operators by smooth elliptic factors and alters neither the order
nor the canonical relation
\cite[\S6]{guillemin_GeometricAsymptotics_1977},
\cite[App.~A]{quinto_IntroductionXrayTomography_2006}.

\emph{FIO structure, order, ellipticity.}  The oscillatory representation
\eqref{eq:radon-oscillatory} exhibits $\mathcal R$ as a Fourier integral operator
with a single phase variable ($N=1$), phase $\varphi=\tau(b-w\cdot x)$ homogeneous
of degree one, and amplitude $\equiv1$ (so $m=0$).  The phase is nondegenerate,
since $\partial_\tau\varphi=b-w\cdot x$ has $d_{w,b,x}(\partial_\tau\varphi)$ with
$db$-component $1$, hence nonzero, on $C_\varphi=\{b=w\cdot x\}$.  The order formula
\eqref{eq:fio-order} with $n_X=n_Y=d$ gives $\mu_{\mathcal R}=-(d-1)/2$, and the
nonvanishing amplitude makes $\mathcal R$ elliptic.

\emph{Canonical relation.}  Identify $T_w^*\Sph^{d-1}\cong T_w\Sph^{d-1}$ by the
round metric and write $x_w^\top:=x-(x\cdot w)w$.  Differentiating $\varphi$ gives
$d_b\varphi=\tau$, $d_w\varphi=-\tau x_w^\top$, and $-d_x\varphi=\tau w$, so
\begin{align}
    C_{\mathcal R}
    =\bigl\{(w,b;-\tau x_w^\top,\tau;\ x,\tau w):\ b=w\cdot x,\ \tau\neq0\bigr\}
    \subset(T^*Y\setminus0)\times(T^*\R^d\setminus0).
    \label{eq:radon-canonical-relation}
\end{align}

\emph{Local-graph property.}  The input covector is $\xi=\tau w$, so for fixed
$(x,\xi)$ there are exactly the two antipodal solutions
$(w,\tau)=(\xi/\abs\xi,\abs\xi)$ and $(-\xi/\abs\xi,-\abs\xi)$; a conic cutoff
selecting one sign of $\tau$ removes the second branch.  On $\tau>0$ the relation
\eqref{eq:radon-canonical-relation} is the graph of the homogeneous canonical
transformation
\begin{align*}
    (x,\xi)\ \longmapsto\
    \Bigl(\tfrac{\xi}{\abs\xi},\ x\cdot\tfrac{\xi}{\abs\xi}\,;\
    -\abs\xi\,x_w^\top,\ \abs\xi\Bigr),
    \qquad w=\tfrac{\xi}{\abs\xi},
\end{align*}
with smooth inverse $x=bw-\eta_w/\eta_b$, $\xi=\eta_b w$ (here $\eta_w,\eta_b$ are
the $T^*Y$ components $-\tau x_w^\top,\tau$); the branch $\tau<0$ is identical with
$\eta_b<0$.  Each branch is therefore a local canonical graph, so every compactly
cut-off piece $\chi_Y\mathcal R\chi_X$ meets the geometric hypothesis of
Theorem~\ref{thm:sss-fio-lp}.

\emph{Natural projection rank.}
The sharpness part of Theorem~\ref{thm:sss-fio-lp} uses the natural projection of
$\mathcal{C}_{\mathcal R}$ to the base variables $(w,b,x)\in Y\times\R^d$.  On the branch
$\tau>0$, parametrize the relation by $(x,w,\tau)$ with $w\in\Sph^{d-1}$ and
$\tau>0$:
\begin{align*}
    (x,w,\tau)
    \longmapsto
    (w,b,x)=(w,w\cdot x,x).
\end{align*}
Its differential sends $(\delta x,\delta w,\delta\tau)$ to
\begin{align*}
    \bigl(\delta w,\, w\cdot\delta x+x\cdot\delta w,\,\delta x\bigr).
\end{align*}
If this vanishes, then $\delta w=0$ and $\delta x=0$, while $\delta\tau$ is free.
Thus the kernel is exactly the radial direction in the conic fibre and the rank is
$2d-1$, the maximal possible rank for a conic canonical relation.  The branch
$\tau<0$ is identical, and transposition gives the same rank statement for
$\mathcal{C}_{\mathcal R}^t$.  By Section~\ref{app:fio-composition}, $\mathrm{A}^k$ and the
synthesis map carry the same phase as $\mathcal R$ and $\mathcal R^*$, hence the same
canonical relations, and inherit this maximal-rank condition verbatim.

\emph{The dual transform.}  $\mathcal R^*$ has the transposed Schwartz kernel
$\delta(b-w\cdot x)$, whence $\mathcal R^*\in I^{-(d-1)/2}(\R^d,Y;\mathcal{C}_{\mathcal R}^t)$
with the transposed relation $\mathcal{C}_{\mathcal R}^t$, the graph of $\chi^{-1}$.  Since
the transpose of a local canonical graph is again one, the cut-off pieces
$\chi_X\mathcal R^*\chi_Y$ satisfy the same hypothesis.

Both $\mathcal R$ and $\mathcal R^*$ are thus elliptic Fourier integral operators
of order $-(d-1)/2$ whose canonical relations, after the conic branch cutoff, are
local canonical graphs, which is the assertion of Lemma~\ref{lem:radon-fio}.
\end{proof}

\subsection{Oscillatory kernels of the analysis and synthesis maps}
\label{app:fio-composition}

We give more details of Lemma~\ref{lem:inverse-fio-pair}.  Both $\mathrm{A}^k$ and
the synthesis map carry the phase of $\mathcal R$ itself, and the multiplier they carry
in the bias variable enters as an \emph{amplitude}, whose order and ellipticity can be
read off directly.

\paragraph{The two kernels.}
Write $m(D_b):=\mathcal F_b^{-1}m\,\mathcal F_b$ and
\begin{align}
    m_A(\tau):=(i\tau)^{k+1}\abs{\tau}^{d-1},
    \qquad
    m_S:=\mathcal F\sigma_k ,
    \label{eq:two-multipliers}
\end{align}
so that $\mathrm{A}^k=c_d\,m_A(D_b)\mathcal R$ by Definition~\ref{def:canonical-density}
and $v_g=\mathcal R^{*}m_S(D_b)g^{R}$ by \eqref{eq:AR-as-composition}.  Inserting the slice theorem \eqref{eq:Fourier-slice} into \eqref{eq:radon-oscillatory}
and into its transpose, the Schwartz kernels are, as identities in $\mathcal S'$,
\begin{align}
    K_{\mathrm A}(w,b;x)
    &=\frac{c_d}{2\pi}\int_{\R}e^{\,i\tau(b-w\cdot x)}\,m_A(\tau)\,d\tau,
    \label{eq:analysis-kernel}\\
    K_{\mathrm S}(x;w,b)
    &=\frac{1}{2\pi}\int_{\R}e^{\,i\tau(w\cdot x-b)}\,m_S(\tau)\,d\tau ,
    \label{eq:synthesis-kernel}
\end{align}
the integrals being the inverse partial Fourier transforms in $\tau$ of the tempered
distributions $c_dm_A$ and $m_S$.  These are not yet oscillatory integrals in the sense
of \eqref{eq:fio-kernel}: $m_S=\mathcal F\sigma_k$ is a distribution and not a symbol.  
The phase function is the Radon phase $\varphi=\tau(b-w\cdot x)$, 
resp.\ its negative, with the single phase
variable $N=1$, nondegenerate with critical set $\{b=w\cdot x\}$ exactly as in
Section~\ref{app:radon-fio-proof}.  The splitting below turns each kernel into two
genuine oscillatory integrals, carrying the two branches of $\mathcal{C}_{\mathcal R}$,
resp.\ $\mathcal{C}_{\mathcal R}^{t}$, plus one smooth kernel.

\paragraph{Two cutoffs, three pieces.}
Since $\partial_t^{k+1}\sigma_k=\delta_0$ gives $(i\tau)^{k+1}m_S=1$ in
$\mathcal S'(\R)$,
\begin{align}
    m_S(\tau)=(i\tau)^{-(k+1)},
    \qquad \tau\neq0 .
    \label{eq:mS-off-origin}
\end{align}
Both $m_A$ and $m_S$ are thus smooth and nonvanishing on each open cone
$\{\pm\tau>0\}$, positively homogeneous there of degrees $k+d$ and $-(k+1)$, and
singular only at $\tau=0$.  Fix $\chi_0\in C_c^\infty(\R)$ with $\chi_0\equiv1$ on
$[-1,1]$ and $\supp\chi_0\subset[-2,2]$, and split, for $m\in\{m_A,m_S\}$,
\begin{align}
    m=a_++a_-+\chi_0m,
    \qquad
    a_\pm:=\mathbf 1_{\{\pm\tau>0\}}\,(1-\chi_0)\,m .
    \label{eq:amplitude-split}
\end{align}
Because $1-\chi_0$ vanishes near the origin, \eqref{eq:mS-off-origin} makes
$(1-\chi_0)m_S$ the smooth function $(1-\chi_0)(i\tau)^{-(k+1)}$, so the products in
\eqref{eq:amplitude-split} are well defined although $m_S$ is only a distribution: any
part of $m_S$ supported at $\tau=0$ is carried by $\chi_0m_S$.  Thus
$a_\pm\in C^\infty(\R)$, and \eqref{eq:amplitude-split} is an exact identity in
$\mathcal S'(\R)$.  The three terms play three distinct roles and the same is true if 
$\chi_0$ is replaced by any other such cutoff, which changes each piece by a smooth kernel.

\emph{$a_\pm$: two elliptic FIOs on single canonical graphs.}  Each $a_\pm$ is a
classical symbol of order $\deg m$ supported in $\{\pm\tau>0\}$, positively homogeneous
there for $\abs\tau\ge1$, with $\abs{a_\pm(\tau)}=\abs{\tau}^{\deg m}\neq0$ on its
support for $\abs\tau\ge1$: it is elliptic, on the whole cone and not merely at one
point.  Substituting $a_\pm$ for $m$ in \eqref{eq:analysis-kernel} or
\eqref{eq:synthesis-kernel} restricts the canonical relation to the branch
$\pm\tau>0$, which is a local canonical graph by Section~\ref{app:radon-fio-proof}.
This is the concrete form of the ``conic cutoff'' of Lemma~\ref{lem:radon-fio}: a
cutoff in $\tau$, that is, a Fourier multiplier in $b$.  The order formula
\eqref{eq:fio-order}, with $N=1$ and $n_X=n_Y=d$, gives
\begin{align*}
    \underbrace{(k+d)}_{\deg m_A}+\tfrac12-\tfrac d2=s_k,
    \qquad
    \underbrace{-(k+1)}_{\deg m_S}+\tfrac12-\tfrac d2=-s_k .
\end{align*}

\emph{$\chi_0m$: a smoothing operator.}  The distribution $\chi_0m$ is compactly
supported, so by Paley--Wiener--Schwartz
$\kappa:=\mathcal F^{-1}(\chi_0m)$ is smooth with at most polynomial growth.  The
corresponding kernels in \eqref{eq:analysis-kernel}--\eqref{eq:synthesis-kernel} are
$c_d\,\kappa(b-w\cdot x)$ and $\kappa(w\cdot x-b)$, smooth in all variables.  On the
bounded window only $b=w\cdot x$ with $x$ in a compact set is sampled, so H\"older's
inequality on $\MR$ gives, for every $N$ and every compact $K\subset\R^d$,
\begin{align*}
    \norm{\mathcal R^{*}\bigl[(\chi_0m_S)(D_b)\,g^{R}\bigr]}_{C^N(K)}
    \le C_{N,K}\,\norm{g}_{L^p(\MR)},
\end{align*}
and likewise $c_d(\chi_0m_A)(D_b)\mathcal R$ maps $L^1(B_A(0))$ into $C^N$ of any
compact subset of $Y$.  When $d$ is odd this term is inessential: $m_A$ is then a
polynomial and $m_A(D_b)$ a differential operator.

\paragraph{The one composition that is needed.}
The sharpness argument of Section~\ref{app:sharpness} composes two Fourier integral
operators.  If $\mathcal{C}_1=\mathrm{graph}(\chi_1)$ and $\mathcal{C}_2=\mathrm{graph}(\chi_2)$
are local canonical graphs whose composition is transversal, then
\begin{align}
    I^{\mu_1}(Z,Y;\mathcal{C}_1)\circ I^{\mu_2}(Y,X;\mathcal{C}_2)
    \subset I^{\mu_1+\mu_2}\bigl(Z,X;\mathrm{graph}(\chi_1\circ\chi_2)\bigr)
    \label{eq:fio-fio-composition}
\end{align}
\cite[Ch.~6]{sogge_FourierIntegralsClassical_2017},
\cite{hormander_LagrangianDistributionsFourier_2009}.  Composing a graph with the
transpose of the same graph gives the diagonal $\Delta_{T^*Y}$.  Hence pairing the
branch $a_+$ of $\mathrm{A}^k$ with the branch $a_+$ of the synthesis map, and $a_-$
with $a_-$, produces operators in $\Psi^0(Y)$, bounded on every $L^p(Y)$,
$1<p<\infty$, by Calder\'on--Zygmund theory.  The two mixed pairs produce order-zero
Fourier integral operators whose canonical relation is the lift of the antipodal
diffeomorphism $(w,b)\mapsto(-w,-b)$ of $Y$; each is a pseudodifferential operator
composed with a pullback by a diffeomorphism, hence also bounded on every $L^p(Y)$ with
no Seeger--Sogge--Stein loss.  Order zero by itself would not give this: a general
order-zero Fourier integral operator loses $\delta_p$.  The sum of the four pieces is
the operator $B$ of Section~\ref{app:sharpness}.

\subsection{Sharpness of the sandwich}
\label{app:sharpness}

We prove Proposition~\ref{prop:radon-lp-sandwich-sharp} by reducing each half of the
sandwich, via the assumed improved embedding, to an $L^p$--Sobolev bound for one of the
FIOs $\mathrm A^k$ (order $s_k$) or $\mathcal R^{*}m_S(D_b)$ (order $-s_k$), localized to a
single branch of $\mathcal C_{\mathcal R}$ resp.\ $\mathcal C_{\mathcal R}^{t}$, and then
invoking the sharpness part of Theorem~\ref{thm:sss-fio-lp}.

\paragraph{A wave-packet example.}
The loss is also attained by explicit extremizers: in coordinates adapted to one branch,
the plate packets $f_\lambda(x)=e^{i\lambda x_d}\phi(\lambda^{\beta}x')\zeta(x_d)$, with
$\beta=0$ for $p\ge2$ and $\beta=1$ for $1<p\le2$, satisfy
$\norm{\mathcal Rf_\lambda}_{W^{\rho,p}(Y)}/\norm{f_\lambda}_{L^p(\R^d)}\sim\lambda^{\rho-\rho_p}$,
so that no gain beyond $\rho_p$ is admissible; see
\cite{peral_LpEstimatesWave_1980,seeger1991regularity,sogge_FourierIntegralsClassical_2017}.
These families are not needed below, where sharpness is used only through
Theorem~\ref{thm:sss-fio-lp}.

\paragraph{Cutoffs.}
Fix any point of the branch $\tau>0$, with base points $x_0\in\Omega$ on $\R^d$ and
$y_0\in\operatorname{int}(\MR)$ on $Y$. Choose $\chi,\eta\in C_c^\infty(\Omega)$ with
$\chi=1$ near $x_0$ and $\eta=1$ on $\supp\chi$, and $\psi\in C_c^\infty(Y)$ equal to $1$
near $y_0$ and supported away from the antipode $(-y_0)$. Choose zeroth-order cutoffs
$Q_X\in\Psi^0(\R^d)$ and $Q_Y\in\Psi^0(Y)$, both properly supported and microlocally
selecting the branch $\tau>0$ near the fixed point, with the range of $Q_X$ contained in
$\{\chi=1\}$ and the range of $Q_Y$ contained in $\operatorname{int}(\MR)$. For $\sigma\ge0$,
\begin{align}
    \norm{\chi F}_{W^{\sigma,p}(\R^d)}\le C\norm{F}_{W^{\sigma,p}(\Omega)},
    \qquad
    \norm{v}_{W^{\sigma,p}(\Omega)}\le\norm{v}_{W^{\sigma,p}(\R^d)},
    \label{eq:sharp-norm-transfer}
\end{align}
the first by the extension operator of the Lipschitz domain $\Omega$ and $\supp\chi\subset\Omega$,
the second by definition of the restriction norm.

\paragraph{Right side: $r_{k,p}$ cannot be raised.}
Suppose \eqref{eq:no-right-improvement} held, and set $\sigma:=r_{k,p}+\varepsilon$. Since the
range of $Q_Y$ lies in $\operatorname{int}(\MR)$, every $Q_Yf$ satisfies $(Q_Yf)^{R}=Q_Yf$, so
the truncation is invisible and $\mathrm S^k_RQ_Yf=\bigl(\mathcal R^{*}m_S(D_b)Q_Yf\bigr)|_\Omega$.
The quotient bound $\norm{\mathrm S^k_RQ_Yf}_{\mathcal RL^p_k(\Omega)}\le\norm{Q_Yf}_{L^p(\MR)}$,
the assumed embedding, the $L^p$-boundedness of $Q_Y$, and \eqref{eq:sharp-norm-transfer} give
\begin{align*}
    \norm{T_{\mathrm S}f}_{W^{\sigma,p}(\R^d)}\le C\norm{f}_{L^p(Y)},
    \qquad
    T_{\mathrm S}:=\chi\,\mathcal R^{*}m_S(D_b)\,Q_Y .
\end{align*}
Modulo a smooth kernel $T_{\mathrm S}$ is an elliptic graph FIO of
order $-s_k$ carrying the branch $\tau>0$, with maximal projection rank.
Theorem~\ref{thm:sss-fio-lp} therefore caps its output smoothness on $L^p=W^{0,p}$ at
$0-(-s_k)-\delta_p=r_{k,p}<\sigma$, a contradiction. This proves \eqref{eq:no-right-improvement}.

\paragraph{Left side: $s_{k,p}$ cannot be lowered.}
Suppose \eqref{eq:no-left-improvement} held, and let $v\in C_c^\infty(\{\chi=1\})$. By the
quotient-norm definition there is $h\in L^p(\MR)$ with $\mathrm S^k_Rh=v$ on $\Omega$ and
$\norm{h}_{L^p(\MR)}\le 2C\norm{v}_{W^{s_{k,p}-\varepsilon,p}(\Omega)}$. Since $\eta=1$ on
$\supp v\subset\Omega$ and $\eta\in C_c^\infty(\Omega)$, one has $v=\eta\,\mathrm S^k_Rh$ on all
of $\R^d$; this global identity matters because $\mathrm A^k$ is nonlocal. Hence
$\psi\,\mathrm A^kv=Bh$ with $B:=\psi\,\mathrm A^k\eta\,\mathrm S^k_R$, which by
Lemma~\ref{lem:inverse-fio-pair} and \eqref{eq:fio-fio-composition} is bounded on $L^p(Y)$.
With \eqref{eq:sharp-norm-transfer},
\begin{align*}
    \norm{\psi\,\mathrm A^kv}_{L^p(Y)}\le C\norm{v}_{W^{s_{k,p}-\varepsilon,p}(\R^d)},
    \qquad v\in C_c^\infty(\{\chi=1\}).
\end{align*}
Since $Q_X$ has range in $\{\chi=1\}$ and is bounded on every $W^{t,p}(\R^d)$, applying this to
$v=Q_Xf$ gives
\begin{align*}
    \norm{T_{\mathrm A}f}_{L^p(Y)}\le C\norm{f}_{W^{s_{k,p}-\varepsilon,p}(\R^d)},
    \qquad
    T_{\mathrm A}:=\psi\,\mathrm A^kQ_X .
\end{align*}
Modulo a smooth kernel $T_{\mathrm A}$ is an elliptic graph FIO of order $s_k$ carrying the
branch $\tau>0$, with maximal projection rank. For input smoothness $s_{k,p}-\varepsilon$,
Theorem~\ref{thm:sss-fio-lp} gives the critical output smoothness
$(s_{k,p}-\varepsilon)-s_k-\delta_p=-\varepsilon<0$, ruling out the displayed
$L^p=W^{0,p}$ bound. This proves \eqref{eq:no-left-improvement} and completes the proof of
Proposition~\ref{prop:radon-lp-sandwich-sharp}.

\newpage
\bibliographystyle{plain}
\bibliography{ref.bib}

\end{document}